\documentclass[12pt]{amsart}
\usepackage{pstricks}
\usepackage{amssymb, a4, mathrsfs}
\newpsobject{showgrid}{psgrid}{subgriddiv=1,griddots=5,gridlabels=6pt}
\usepackage[top=2.25cm, bottom=2.5cm, left=2.5cm, right=2.5cm ]{geometry}
\usepackage{paralist}
\usepackage{bbding}
\usepackage{graphicx}
\usepackage{url}
\usepackage{fancyhdr}
\usepackage{footmisc}
\usepackage{algorithmic}
\usepackage{listings}
\usepackage{fancyvrb}
\usepackage{cite}

\usepackage{tikz}

\usetikzlibrary{arrows.meta}
\usepackage{amscd}

\usepackage{indentfirst}
\usepackage{lineno}

\usetikzlibrary{arrows,automata}

\newtheorem{them}{Theorem}[section]
\newtheorem{prop}[them]{\noindent Proposition}
\newtheorem{lemma}[them]{\noindent Lemma}
\newtheorem{corollary}[them]{\noindent Corollary}

\theoremstyle{definition}
\newtheorem{definition}[them]{\noindent Definition}

\begin{document}

\title[Network right $*$-abundant semigroups]{Network right $*$-abundant semigroups}
\thanks{ }
\author{Yanhui Wang}
\address{College of Mathematics and Systems Science\\ Shandong University of Science and Technology\\ Qingdao 266590\\ China}
\email{yanhuiwang@sdust.edu.cn}
\author{Pei Gao}
\address{College of Mathematics and Systems Science\\ Shandong University of Science and Technology\\ Qingdao 266590\\ China}
\email{gaopei0403@163.com}
\author{Xueming Ren}
\address{Department of Mathematics\\ Xi'an University of Architecture and Technology\\ Xi'an 710055\\ China}
\email{xmren@xauat.edu.cn}
\noindent
\maketitle

\renewcommand{\thefootnote}{\empty}

\vspace{3mm}
\begin{abstract}
We introduce the class of network right $*$-abundant semigroups. These are based on networks that extend the notion of a directed graph. This class properly contains  the class of graph inverse semigroups. We investigate the structure of network right $*$-abundant semigroups. We show that two network right $*$-abundant semigroups are isomorphic if and only if the underlying networks are isomorphic.
\end{abstract}

\section{Introduction}\label{Intro}
 In this article,  a {\it network} $\Gamma = (V, T)$ consists of a set $V$ of vertices and a set $T$ of relations.  Each relation in $T$ consists of an ordered pair of disjoint non-empty subsets of $V$. If each $t \in T$ is an ordered pair of distinct singleton subsets of $V$, then we identify $\Gamma$ with the underlying simple directed graph and refer it simply as a {\it graph}.   There exist two maps $\mathbf{s}$ and $\mathbf{r}$ from $T$ to the power set ${\mathcal{P}}(V)$ of $V$. For all $t \in T$, $\mathbf{s}(t)$ and $\mathbf{r}(t)$ are called the {\it source} and {\it range} of $t$, respectively. Networks under consideration in this paper will have finitely many or countably infinitely many vertices and  relations.

In a graph $\Gamma = (V, T)$, a {\it path} is a finite sequence $p = t_1t_2\cdots t_n$ of relations $t_i \in T$ with $\mathbf{r}(t_{i}) =\mathbf{s}(t_{i+1})$ for $i = 1, \ldots, n-1$. In the 1970s, Ash and Hall~\cite{2} constructed inverse semigroups using paths in a graph, called {\it graph inverse semigroups}. They showed that every partial order can be realized as that of non-zero ${\mathcal{J}}$-classes of an inverse semigroup. Graph inverse semigroups not only generalize polycyclic monoids~\cite{19} but also arise in the study of rings and $C^*$-algebras~\cite{21}. In particular,  \cite{3,4} introduced the Cohn path $K$-algebra of $\Gamma$ and  Leavitt path $K$-algebra of $\Gamma$ where $K$ is a field and $\Gamma$ is a graph. There exists a kind of paths  $p = t_1t_2\cdots t_n$ of relations $t_i \in T$ with either $\mathbf{r}(t_{i}) = \mathbf{s}(t_{i+1})$  or $\mathbf{r}(t_{i}) \neq \mathbf{s}(t_{i+1})$  but $\mathbf{r}(t_{i}) \cap \mathbf{s}(t_{i+1}) \neq \emptyset$ for $i = 1, \ldots, n-1$ in a network.  Naturally,  one may ask what kind of algebraic systems such paths may lead to? In this paper we construct from a network a right $*$-abundant semigroup with zero, called a {\it network right $*$-abundant semigroup}, as a quotient of a semigroup of paths and their inverse. If our network is a graph, then our network right $*$-abundant semigroup is the corresponding graph inverse semigroup.

 The structure of this paper is as follows. In Section~\ref{sec:Pre} we begin with some basic definitions and results concerning  semigroups, in particular, right abundant semigroups. We also give the definition and properties of paths in a   network and homomorphisms of  networks. Section~\ref{sec:NRAS}  constructs a right $*$-abundant semigroup $Q_{\Gamma}$ using paths in a network $\Gamma$, named {\it network right $*$-abundant semigroups}.  Dually, we can get a network left abundant semigroup. We build a right ample  subsemigroup $S_{\Gamma}$ of $Q_\Gamma$, and we also  show that the inverse semigroup $R_{\Gamma}$ generated by linear paths is a fundamental inverse subsemigroup of $S_{\Gamma}$.  Section~\ref{sec:congruence} gives a proper ideal $I$ of $Q_{\Gamma}$ (resp. $S_{\Gamma}$) and also $\rho_I = (I \times I) \cup 1_{Q_{\Gamma}}$ is an idempotent-separating congruence on $Q_{\Gamma}$.  A sufficient condition is given under which $Q_\Gamma$ is not $*$-congruence-free as a unary semigroup. In addition, we also find a sufficient condition under which $S_{\Gamma}$ is not $*$-congruence-free as a unary semigroup and  $R_{\Gamma}$ is not  congruence-free. Section~\ref{sec:hom} describes properties of idempotents of $Q_{\Gamma}$ in terms of the natural partial order in a semigroup.   We also  show that two  networks are isomorphic if and only if their corresponding network right $*$-abundant  semigroups are isomorphic. Section~\ref{sec:ex} gives an example to explain that $Q_{\Gamma}$ contains properly $S_{\Gamma}$ and $R_{\Gamma}$.

\section{Preliminaries}\label{sec:Pre}
To make this article self-contained  we recall some basic definitions and results concerning  right abundant semigroups and networks. For more details, we refer the reader to ~\cite{Bower} and \cite{F, Fountain, Howie}.

\subsection{(Right) abundant semigroups}
 We begin with relations ${\mathcal{L}}^{*}$ and ${\mathcal{R}}^{*}$.

Let $S$ be a semigroup. For any $a, b \in S$,
$$a\, {\mathcal{L}}^*\, b \Leftrightarrow \forall x, y \in S^{1} \big(ax = ay \Leftrightarrow bx = by\big)$$
and
$$a\, {\mathcal{R}}^*\, b \Leftrightarrow \forall x, y \in S^{1} \big(xa = ya \Leftrightarrow xb = yb\big).$$
Clearly, ${\mathcal{L}}^*$ is a right congruence and ${\mathcal{R}}^*$ is a left congruence.

\begin{lemma}\label{L*Def}
 If $e$ is an idempotent of a semigroup $S$, then the following statements are equivalent for $a\in{S}$:
\begin{enumerate}
\item[\rm (i)] $(e,a)\in{\mathcal{L^*}}$;
\item[\rm (ii)] $ae=a$ and for all $x,y\in{S^{1}}$, $ax=ay$ implies $ex=ey$.
\end{enumerate}
\end{lemma}

The dual to Lemma~\ref{L*Def} holds for the relation ${\mathcal{R}}^*$.

We pause to mention that ${\mathcal{L}}\subseteq {\mathcal{L}}^*$  and ${\mathcal{R}}\subseteq{\mathcal{R}}^*$ on any semigroup $S$. For any regular elements $a,b\in{S}$, $(a,b)\in{\mathcal{L}}^*$ if and only if $(a,b)\in{\mathcal{L}}$ and $(a,b)\in{\mathcal{R}}^*$ if and only if $(a,b)\in{\mathcal{R}}$. In particular, if $S$ is a regular semigroup, then ${\mathcal{L}}^*={\mathcal{L}}$ and ${\mathcal{R}}^*={\mathcal{R}}$.

A semigroup is  said to be a {\it right abundant} semigroup if each ${\mathcal{L}}^*$-class contains an idempotent.

\begin{definition}
A right abundant semigroup $S$ is said to be a {\it right $*$-abundant} semigroup if each ${\mathcal{L}}^*$-class of $S$ contains a unique idempotent.
\end{definition}

Let $S$ be a right $*$-abundant semigroup with set of idempotents $E(S)$. We denote the unique idempotent of $E(S)$ in the ${\mathcal{L}}^*$-class of $a$ by $a^*$. Then $^*$ is a unary operation on $S$ and we may regard $S$ as an algebra of type $(2, 1)$ and call such algebras {\it unary algebras}; as such, morphisms must preserve the unary operation of $^*$ (and hence the relation ${\mathcal{L}}^*$). We may refer to such morphisms as `$(2, 1)$-morphisms' if there is danger of ambiguity. Of course, any semigroup isomorphism must preserve the additional operations.

\begin{definition}
 A {\it right ample} (formally {\it right type $A$}) semigroup $S$ is defined to be a right abundant semigroup  in which the idempotents commute and for all $a \in S$, $e \in E(S)$, $ea=a(ea)^*$.
\end{definition}

From the commutativity of idempotents it is clear that any ${\mathcal{L}}^*$-class of a right ample semigroup $S$ contains at most one idempotent of $E(S)$. So a right ample semigroup is a right $*$-abundant semigroup. Note that as unary semigroups, the class of right ample semigroups is a quasi-variety of right $*$-abundant semigroups~\cite{Fountain1}.

 Dually, a {\it left ample} (formally {\it left type $A$}) semigroup is defined. An {\it ample} (formally {\it type A}) semigroup is defined to be a left and right ample semigroup. In particular, an inverse semigroup is ample, where $a^{\dagger}= aa^{-1}$ and $a^* = a^{-1}a$, where $a^{\dagger}$ is the unique idempotent in the ${\mathcal{R}}^{*}$-class containing $a$ and  $a^{-1}$ is the inverse of $a$.  Ample semigroups are usually thought of as the appropriate abundant analogue of inverse semigroups.

\begin{lemma}\label{inverse}
If $S$ is a right ample semigroup where $E(S)$ is its semilattice of idempotents, and $H$ is the set of all regular elements of $S$, then $H$ is an inverse subsemigroup of $S$.
\end{lemma}
\begin{proof}
 Clearly $E(S) \subset H$. Let $a, b \in H$. Recall that $a\, {\mathcal{L}}^*\,  a^*$ and $b\, {\mathcal{L}}^*\, b^*$ in $S$. Since $a, b, a^*$ and $b^*$ are regular elements in $S$, then $a\, {\mathcal{L}}\, a^*$, $b\, {\mathcal{L}}\, b^*$ so that as ${\mathcal{L}}^*$ is a right congruence we have  $ab\, {\mathcal{L}}\, a^*b$, which implies that $(ab)^*= (a^*b)^*$. As $S$ is right ample, then $a^*b = b(a^*b)^*$ and then $(ab)^* = (a^*b)^*=b^*(ab)^*= b^*(a^*b)^*$. Therefore, $ab\, {\mathcal{L}}\, a^*b = b(a^*b)^*\, {\mathcal{L}}\, b^*(a^*b)^* = (ab)^*$, that is, $ab$ is $\mathcal{L}$-related to an idempotent, so that $ab$ is a regular element. Hence $H$ is a subsemigroup of $S$ and so the result holds.
\end{proof}

\begin{lemma}\label{aa-1RaLa-1a}\cite{Howie}
Let $S$ be a  semigroup with set of idempotents $E(S)$, and let $a$ be a regular element in $S$ where $b$ is an inverse of $a$. Then $ab, ba \in E(S)$ and $ab\, {\mathcal{R}}\, a \, {\mathcal{L}}\, ba$. Moveover, if $E(S)$ is a semilattice and $a$ is a regular element in $S$ then the inverse of $a$ is unique.
\end{lemma}

\subsection{Generation and presentation}

We now  recall the notion of a semigroup generated by a non-empty set $X$. The free monoid $X^*$ on $X$ consists of all words over $X$ with operation of juxtaposition. We use $\varepsilon$ to denote the empty word. The free semigroup $X^+$ on $X$ is $X^*\backslash\{\varepsilon\}$. A non-empty word is denoted by $x_1x_2\cdots x_n$ where $x_i \in X$ for $1 \leq i \leq n$. For any two words $\alpha = x_1x_2\cdots x_n$, $\beta =y_1y_2 \cdots y_m$ of $X^+$, we use $\alpha\beta$ to denote the juxtaposition of $\alpha$ and $\beta$, that is $\alpha\beta= x_1x_2\cdots x_ny_1y_2 \cdots y_m$. If $\alpha = \beta\mu$ for any  words $\alpha, \beta, \mu \in X^*$, $\beta$ is called a {\it prefix} of $\alpha$, and a {\it proper prefix} if $\mu$ is not the empty word $\varepsilon$.  For any two non-empty words $\alpha, \beta$ in $X^+$, we say that $\alpha, \beta$ are {\it prefix comparable} if one of $\alpha, \beta$ is a prefix of the other.

   Let $R$ be a binary relation on $X^+$. The quotient semigroup $X^+/R^{\sharp}$ , where $R^{\sharp}$ is the smallest congruence on $X^+$ containing $R$, is said to be {\it the semigroup generated by $X$ subjects to relations $R$}. We cite the formal equality $u = v$ to mean that $(u, v) \in R$.  We denote the $R^{\sharp}$-class of $x\in X^{+}$ in $X^+/R^{\sharp}$ by $[x]$ so the reader should bear in mind the context in each case.

We end the section with definition of reduction systems and their properties. More details are referred to \cite{Book} and \cite{YangG}.

Let $A$ be a set and $\rightarrow$ a binary relation on $A$. We call the structure $(A, \rightarrow)$ a {\it reduction system} and the relation $\rightarrow$ a {\it reduction relation}.  The reflexive, transitive closure of $\rightarrow$ is denoted by $\overset{\ast}{\rightarrow}$, while $\overset{\ast}{\leftrightarrow}$ denotes the smallest equivalence relation on $A$ that contains $\rightarrow$. We denote the equivalence class of an element $x \in A$ by $[x]$. An element $x \in A$ is said to be {\it irreducible} or {\it reduced} if there is no $y \in A$ such that $x \rightarrow y$; otherwise, $x$ is {\it reducible}. For any $x, y \in A$, if $x \overset{\ast}{\rightarrow} y$ and $y$ is irreducible, then $y$ is a {\it normal form} of $x$. A reduction system $(A, \rightarrow)$ is {\it noetherian} if there is no infinite sequence $x_0, x_1, \ldots \in A$ such that for all $i \geq 0$, $x_i \rightarrow x_{i+1}$.

We say that a reduction system $(A, \rightarrow)$ is {\it confluent} if whenever $w, x, y \in A$ are such that $w \overset{\ast}{\rightarrow} x$ and $w \overset{\ast}{\rightarrow} y$, then there is a $z \in A$ such that $x \overset{\ast}{\rightarrow} z$ and $y \overset{\ast}{\rightarrow} z$; and $(A, \rightarrow)$ is {\it locally confluent} if whenever $w, x, y \in A$, are such that $w \rightarrow x$ and $w \rightarrow y$, then there is a $z \in A$ such that $x \overset{\ast}{\rightarrow} z$ and $y \overset{\ast}{\rightarrow} z$.

\begin{lemma}\cite{Book}\label{rwriting}
Let $(A, \rightarrow)$ be a reduction system.

{\rm (i)} If $(A, \rightarrow)$ is noetherian and confluent, then for each $x \in A$, $[x]$ contains a unique normal form.

{\rm (ii)} If $(A, \rightarrow)$ is noetherian, then it is confluent if and only if it is locally confluent.
\end{lemma}

Let $S$ be a semigroup with presentation $\langle X: u_i = v_i, i \in I\rangle$, where $u_i, v_i \in X^+$. We can form a reduction system $(X^+, \rightarrow)$ where
\[u \rightarrow v \Leftrightarrow (u = xu_iy, v = xv_iy\ \mbox{for\ some}\ x, y \in X^{\ast}, i \in I).\]
It is clear that $\overset{\ast}{\leftrightarrow}$ is the congruence generated by $R = \{(u_i, v_i): i \in I\}$. Thus if $\rightarrow$ is a confluent noetherian rewriting system then every element of $S$ has a unique normal form as a word in $X^+$.

\subsection{Networks}\label{subsec:HOLN}
 In this subsection we give some basic definitions and results of  networks. For further details, of both background and
technicalities, we refer the reader to \cite{HN, Bower} and \cite{Wang}.

\begin{definition}
A network   $\Gamma = (V, T, \mathbf{s}, \mathbf{r})$ consists of the set of vertices $V$, the  set of relations $T$, together with mappings $\mathbf{s}, \mathbf{r} : T \rightarrow {\mathcal{P}}(V)$, respectively called the {\it source mapping} and the {\it range mapping} for $\Gamma$, where ${\mathcal{P}}(V)$ is the power set of $V$ and for all $t \in T$, $\mathbf{s}(t)$ and $\mathbf{r}(t)$ are disjoint non-empty subsets of $V$.
\end{definition}

As we stated in Introduction if for all $t \in T$ $\mathbf{s}(t)$ and $\mathbf{r}(t)$ are singletons in a network $\Gamma = (V, T, \mathbf{s}, \mathbf{r})$ we identify $\Gamma$ with the underlying simple direct graph and refer to it simply as a {\it graph}.

 Let $\Gamma$ be a   network with set of vertices $V$ and set of relations $T$. Put
 $$T^0=\{A \subseteq V:\exists\ t\in T, A=\mathbf{s}(t)\ \mbox{or}\ A=\mathbf{r}(t)\} \cup V,$$
and for all $A \in T^0$, we set $\mathbf{s}(A) = A=\mathbf{r}(A)$.

  Here we should remark that if $\Gamma$ is a graph,
 $$T^0= V,$$
 as the source and target of a relation are singleton sets, respectively.

 \begin{definition}\label{path}
 A {\it path} in a  network $\Gamma = (V, T, \mathbf{s}, \mathbf{r})$ is a finite sequence $\alpha = t_1t_2\cdots{t_n}$ of elements of $T \cup T^0$ such that $\mathbf{r}(t_i)\cap{\mathbf{s}(t_{i+1})}\ne\emptyset$ for $i=1,2,\cdots,n-1$. In such a case, $\mathbf{s}({\alpha}) = \mathbf{s}(t_1)$ is the {\it source} of $\alpha$, $\mathbf{r}(\alpha) = \mathbf{r}(t_{n})$ is the {\it range} of $\alpha$.
 \end{definition}

An element in $T^0$ is said to be an {\it empty path}. Let $P(\Gamma)$ denote the set of all paths in a network $\Gamma$. This includes the zero-element $0$ and $T^0$. That is,
 $$P({\Gamma}) = \{t_1t_2\cdots{t_n}: t_i \in T\ \cup\ T^0, \mathbf{r}(t_i)\ \cap\ \mathbf{s}(t_{i+1}) \ne \emptyset\ \mbox{for}\ i=1, 2,\cdots, n-1 \}\ \cup\  \{0\}.$$

 \begin{definition}\label{linePath}
If $\mathbf{r}(t_i)=\mathbf{s}(t_{i+1})$ for $i=1,2,\cdots,n-1$ in $\alpha=t_1t_2\cdots{t_n}\in{P(\Gamma)}$, then $\alpha$ is said to be a
 {\it linear path}.
 \end{definition}

  Let $LP(\Gamma)$ denote the set of all linear paths in $\Gamma$. This includes  $T^0$ but does not contain the zero element $0$. That is,
  $$LP(\Gamma) = \{t_1t_2\cdots{t_n}\in P(\Gamma)\setminus\{0\}: \mathbf{r}(t_i)=\mathbf{s}(t_{i+1}) \ \mbox{for} \ i=1,2,\cdots,n-1 \}.$$

It is a good place to remark that a non-zero path in a graph must be a linear path as the source and target of a relation are singleton sets, respectively. Hence  we have $P(\Gamma) = LP(\Gamma) \cup \{0\}$ in a graph $\Gamma$.

For any two sequences $\alpha=t_1t_2\cdots{t_n}$ and $\beta=y_1y_2\cdots{y_m}$ of elements of $T\ \cup\ T^0$ we denote the sequence $t_1t_2\cdots{t_n}y_1y_2\cdots{y_m}$ by $\alpha\beta$; then $\alpha\beta$ is a path if and only if $\mathbf{r}(\alpha)\cap{\mathbf{s}(\beta)}\neq \emptyset$, that is, $\mathbf{r}(t_n)\cap{\mathbf{s}(y_1)}\neq \emptyset$, and $\alpha$ and $\beta$ are paths.  The sequence $\alpha\beta$ is a line path if and only if $\alpha$, $\beta$ are line paths and $\mathbf{r}(\alpha) = {\mathbf{s}(\beta)}$. So in general, the juxtaposition of two line paths is possible a path but not a line path. Moreover, we have the following result.

\begin{lemma}\label{SubsemigroupLP}
The set $P(\Gamma)$ together with the operation $\circ$ that for all $\alpha, \beta \in P(\Gamma)$,
$$ \alpha \circ \beta  =
	\begin{cases}
		\alpha\beta & \mbox{if}\ \mathbf{r}(\alpha) \cap \mathbf{s}(\beta) \neq \emptyset\\
        0 & \mbox{otherwise},
	\end{cases}$$
forms a semigroup with zero.
\end{lemma}

  According to the definition of linear paths, it is easy to get the following Lemma so we omit its proof.

\begin{lemma}\label{LP_Property}
 For any $\alpha,\beta\in{LP(\Gamma)}$, we have
\begin{enumerate}
\item[\rm (i)] if $\alpha=t_1t_2\cdots{t_n}\in{LP(\Gamma)}$ then $t_it_{i+1}\cdots{t_j}\in{LP(\Gamma)}$ for $1\leq{i} < j\leq{n}$;
\item[\rm (ii)] if $\mathbf{r}(\alpha)=\mathbf{s}(\beta)$ then $\alpha\beta\in{LP(\Gamma)}$.
\end{enumerate}
 \end{lemma}

\begin{definition}
A {\it homomorphism} $\phi = (\phi_V, \phi_T)$ from $\Gamma =(V_{\Gamma}, T_{\Gamma}, \mathbf{s}, \mathbf{r})$ to $\Delta = (V_{\Delta}, T_{\Delta}, \mathbf{s}, \mathbf{r})$ consists of two maps $\phi_V: V_{\Gamma} \rightarrow V_{\Delta}$ and $\phi_T: T_\Gamma \rightarrow T_{\Delta}$ such that  for all $t \in T_{\Gamma}$  $\mathbf{s}(t)\phi = \{v\phi_V: v \in \mathbf{s}(t)\} = \mathbf{s}(t\phi_T)$ and $\mathbf{r}(t)\phi = \{v\phi_V: v \in \mathbf{r}(t)\} = \mathbf{r}(t\phi_T)$.
\end{definition}

A mapping $\phi = (\phi_V, \phi_T)$ from $\Gamma =(V_{\Gamma}, T_{\Gamma}, \mathbf{s}, \mathbf{r})$ to $\Delta = (V_{\Delta}, T_{\Delta}, \mathbf{s}, \mathbf{r})$, is called {\it bijective} if $\phi_V$ and $\phi_T$ are bijective.  A bijective  homomorphism $\phi$ is an {\it  isomorphism}. Networks $\Gamma$ and $\Delta$ are said to be {\it isomorphic} if there exists an isomorphism from $\Gamma$ to $\Delta$.  It is straightforward to see that the inverse of an isomorphism is also an homomorphism and hence an isomorphism. If two networks  $\Gamma$ and $\Delta$ are  isomorphic we write it as  $\Gamma\cong\Delta$.

 It is a routine to get the following lemma.

\begin{lemma}
If two networks $\Gamma$ and $\Delta$ are isomorphic, then $P(\Gamma)$ is isomorphic to $P(\Delta)$.
\end{lemma}

\section{The semigroup $Q_{\Gamma}$}\label{sec:NRAS}
The aim of this section is to construct a right $*$-abundant semigroup using paths in a  network.

Let $\Gamma =(V, T, \mathbf{s}, \mathbf{r})$ be a network. For any $t \in T \cup T^0$, we define $t^{-1}$ to be a relation with
\[\mathbf{s}(t^{-1}) = \mathbf{r}(t)\ \mbox{and}\ \mathbf{r}(t^{-1}) = \mathbf{s}(t).\]
Notice that $A^{-1} = A$ for each $A \in T^0$.  Put
 $$T^{-1}=\{t^{-1}: t \in T\}.$$

\begin{definition}\label{NRAS}
Let $\Gamma =(V, T, \mathbf{s}, \mathbf{r})$ be a network. The semigroup is given by the presentation $Q_{\Gamma}\ : =\langle X: R\rangle$ where
$$X = T\cup  T^0 \cup T^{-1} \cup \{0\}$$
 and $R$ consists of the following relations:
\begin{enumerate}
\item[\rm (NR1)]\ $\mathbf{s}(t)t=t=t\mathbf{r}(t)$ for $t \in T \cup T^0 \cup\ T^{-1}$;
\item[\rm (NR2)]\ $t_1t_2=0$ if $\mathbf{r}(t_1)\cap \mathbf{s}(t_2)=\emptyset$ for $t_1, t_2 \in T  \cup T^0\ \cup T^{-1}$;
\item[\rm (NR3)]\ $t^{-1}_1t_2=0$ if $t_1\neq{t_2}$ for $t_1, t_2 \in T$;
\item[\rm (NR4)]\ $t^{-1}t=\mathbf{r}(t)$ for $t \in T$;
\item[\rm (NR5)]\ $t^{-1}A=0$ if $\mathbf{s}(t) \neq {A}$ for $t \in T$ and $A \in T^0$;
\item[\rm (NR6)]\ $0x = 0 = x0$ for all $x \in X$.
\end{enumerate}
\end{definition}

It is necessary to remark that $Q_{\Gamma}$ defined in Definition~\ref{NRAS} has a zero element $0$ by (NR6). Here  for $A, B \in T^0$ we regard $AB$ as a path from $A$ to $B$ if $A \cap B \neq \emptyset$.

Notice that it follows from Definition~\ref{path} that a non-zero path $\alpha$ in a  network $\Gamma$ is a finite sequence $\alpha = t_1t_2\cdots{t_n}$ of elements of $T \cup T^0$ such that $\mathbf{r}(t_i)\cap{\mathbf{s}(t_{i+1})}\ne\emptyset$ for $i=1,2,\cdots,n-1$. Then $\alpha$ is actually a word in $(T \cup T^0 \cup T^{-1})^+$. For any path $\alpha= t_1t_2\cdots{t_n} \in P(\Gamma)\setminus \{0\}$, we define
$${\alpha}^{-1}={t_n}^{-1}\cdots{t_2}^{-1}{t_1}^{-1}.$$

We now show that if $\Gamma$ is a graph then $Q_{\Gamma}$ is essentially the graph inverse semigroup. If $\Gamma$ is a graph then  for any path $\alpha = t_1\cdots t_n$ we have $\mathbf{r}(t_i) = \mathbf{s}(t_{i+1})$, $i =1, \ldots, n-1$.  Moreover $T^0= V$ and (NR2) turns into $t_1t_2=0$ if $\mathbf{r}(t_1)\neq \mathbf{s}(t_2)$ for $t_1, t_2 \in T \cup T^0 \cup T^{-1}$ as $\mathbf{r}(t_1)$ and $\mathbf{s}(t_2)$ are singletons, respectively. In this case (NR5) follows from (NR2). Further $Q_{\Gamma}$ is indeed a graph inverse semigroup whose non-zero elements can be written uniquely as $\alpha\beta^{-1}$ where $\alpha$ and $\beta$ are linear paths such that $\mathbf{r}(\alpha) = \mathbf{r}(\beta)$ according to (NR1), (NR2), (NR3), (NR4) and (NR6) \cite{21}. Notice that Ash and Hall \cite{3} first introduced graph inverse semigroups consisting of elements in the style $(\alpha, \beta)$ with $\alpha, \beta \in P(\Gamma)$ and $\mathbf{r}(\alpha) = \mathbf{r}(\beta)$. \cite[Proposition 2]{21} shows that graph inverse semigroups constructed in \cite{3} and \cite{21} are isomorphic.  Here we should stress that each element of $Q_{\Gamma}$ is a congruence class so $\alpha\beta^{-1}$ is actually a  representative  of a class. With this in mind, before we show that $Q_{\Gamma}$ with respect to a network $\Gamma$  is a right abundant semigroup we first show that each element of $Q_{\Gamma}$ has a unique normal form $\alpha\beta^{-1}$ in the same way as  that of a graph inverse semigroup. To do this, we first show that the reduction system $(X^+, \rightarrow)$ where $X = T \cup T^0 \cup T^{-1}$,
\[u \rightarrow v \Leftrightarrow (u = xu_iy, v = xv_iy\ \mbox{for\ some}\ x, y \in X^{\ast}, (u_i, v_i) \in R)\]
 is a confluent rewriting system.

\begin{prop}\label{confluent}
The reduction system $(X^+, \rightarrow)$ where $X = T \cup T^0 \cup T^{-1}$, is a confluent rewriting system.
\end{prop}
 \begin{proof}
 It is sufficient to show  the one-step case $(t_1t_2)t_3 = t_1(t_2t_3)$ for $t_1, t_2, t_3 \in T \cup T^0 \cup T^{-1}$, that is, consider the situation
\begin{center}
\begin{tikzpicture}[x=0.75pt,y=0.75pt,yscale=-1,xscale=1]

\draw    (144.33,58) -- (164.19,86.36) ;
\draw [shift={(165.33,88)}, rotate = 235.01] [color={rgb, 255:red, 0; green, 0; blue, 0 }  ][line width=0.75]    (10.93,-3.29) .. controls (6.95,-1.4) and (3.31,-0.3) .. (0,0) .. controls (3.31,0.3) and (6.95,1.4) .. (10.93,3.29)   ;
\draw    (129.58,57.75) -- (106.56,87.42) ;
\draw [shift={(105.33,89)}, rotate = 307.81] [color={rgb, 255:red, 0; green, 0; blue, 0 }  ][line width=0.75]    (10.93,-3.29) .. controls (6.95,-1.4) and (3.31,-0.3) .. (0,0) .. controls (3.31,0.3) and (6.95,1.4) .. (10.93,3.29)   ;

\draw (117,38) node [anchor=north west][inner sep=0.75pt]   [align=left] {$t_1t_2t_3$};
\draw (96,91) node [anchor=north west][inner sep=0.75pt]   [align=left] {$ut_3$};
\draw (156,91) node [anchor=north west][inner sep=0.75pt]   [align=left] {$t_1v$};
\end{tikzpicture}
\end{center}
we must show in each case that there exists $w$ with
\begin{center}
\begin{tikzpicture}[x=0.75pt,y=0.75pt,yscale=-1,xscale=1]

\draw    (166.33,108) -- (145.36,143.28) ;
\draw [shift={(144.33,145)}, rotate = 300.74] [color={rgb, 255:red, 0; green, 0; blue, 0 }  ][line width=0.75]    (10.93,-3.29) .. controls (6.95,-1.4) and (3.31,-0.3) .. (0,0) .. controls (3.31,0.3) and (6.95,1.4) .. (10.93,3.29)   ;
\draw    (111.58,108.75) -- (128.43,142.21) ;
\draw [shift={(129.33,144)}, rotate = 243.27] [color={rgb, 255:red, 0; green, 0; blue, 0 }  ][line width=0.75]    (10.93,-3.29) .. controls (6.95,-1.4) and (3.31,-0.3) .. (0,0) .. controls (3.31,0.3) and (6.95,1.4) .. (10.93,3.29)   ;

\draw (130,149) node [anchor=north west][inner sep=0.75pt]   [align=left] {$w$};
\draw (96,91) node [anchor=north west][inner sep=0.75pt]   [align=left] {$ut_3$};
\draw (156,90) node [anchor=north west][inner sep=0.75pt]   [align=left] {$t_1v$};
\end{tikzpicture}
\end{center}

The following discussion is based on $t_1t_2$ and $t_2t_3$ satisfying conditions (NR1)-(NR5), respectively. Notice that it never happens that both $t_1t_2$ and $t_2t_3$ satisfy the relations among  (NR3), (NR4) and (NR5). So there exist seven cases as follows:

\vspace{2mm}
Case 1. $t_1t_2$ satisfies (NR1) and $t_2t_3$ satisfies (NR1). Then there exists four cases as follows:

(a1) When $t_1 = \mathbf{s}(t_2)$ and $t_2 = \mathbf{s}(t_3)$ we get $t_1t_2 \rightarrow t_2$ and $t_2t_3 \rightarrow t_3$, further we have $(t_1t_2)t_3 \rightarrow t_2t_3 \rightarrow t_3$ and $t_1(t_2t_3)\rightarrow t_1t_3 \rightarrow t_3$ since $t_2 = \mathbf{r}(t_3) \in T^0$ and then $t_1 = \mathbf{s}(t_2) = t_2 = \mathbf{s}(t_3)$.

(a2) When $t_1 = \mathbf{s}(t_2)$ and $t_3 = \mathbf{r}(t_2)$ we get $t_1t_2 \rightarrow t_2$ and $t_2t_3 \rightarrow t_2$, further we have $(t_1t_2)t_3 \rightarrow t_2t_3 \rightarrow t_2$ and $t_1(t_2t_3)\rightarrow t_1t_2 \rightarrow t_2$.

(a3) When $t_2 = \mathbf{r}(t_1)$ and $t_2 = \mathbf{s}(t_3)$ we get $t_1t_2 \rightarrow t_1$ and $t_2t_3 \rightarrow t_3$, further we have $(t_1t_2)t_3 \rightarrow t_1t_3$ and $t_1(t_2t_3)\rightarrow t_1t_3$, where $\mathbf{r}(t_1) = t_2 = \mathbf{s}(t_3)$.

(a4) When $t_2 = \mathbf{r}(t_1)$ and $t_3 = \mathbf{r}(t_2)$ we get $t_1t_2 \rightarrow t_1$ and $t_2t_3 \rightarrow t_2$, further we have $(t_1t_2)t_3 \rightarrow t_1t_3 \rightarrow t_1$ since $t_2 = \mathbf{r}(t_1) \in T^0$ and then $t_3 = \mathbf{r}(t_2) =t_2 = \mathbf{r}(t_1)$,  and also we have $t_1(t_2t_3)\rightarrow t_1t_2 \rightarrow t_1$.

\vspace{2mm}
Case 2. $t_1t_2$ satisfies (NR1) and $t_2t_3$ satisfies (NR2). Then there exists two cases as follows:

(a1) When $t_1 = \mathbf{s}(t_2)$  and $\mathbf{r}(t_2) \cap \mathbf{s}(t_3) = \emptyset$ we get $t_1t_2 \rightarrow t_2$ and $t_2t_3 \rightarrow 0$, further we have $(t_1t_2)t_3 \rightarrow t_2t_3 \rightarrow 0$ and $t_1(t_2t_3)\rightarrow t_10 \rightarrow 0$.

(a2) When $t_2 = \mathbf{r}(t_1)$ and $\mathbf{r}(t_2) \cap \mathbf{s}(t_3) = \emptyset$ we get $t_1t_2 \rightarrow t_1$ and $t_2t_3 \rightarrow 0$, further we have $(t_1t_2)t_3 \rightarrow t_1t_3 \rightarrow 0$ since $t_2 = \mathbf{r}(t_1) \in T^0$ and then $\mathbf{r}(t_1) \cap \mathbf{s}(t_3)= \mathbf{r}(t_2) \cap \mathbf{s}(t_3) = \emptyset$,  and also we have  $t_1(t_2t_3)\rightarrow t_10 \rightarrow 0$.

Dually, if $t_1t_2$ satisfies (NR2) and $t_2t_3$ satisfies (NR1) then we still get $(t_1t_2)t_3 \rightarrow 0$ and $t_1(t_2t_3)\rightarrow 0$.

\vspace{2mm}

Case 3. $t_1t_2$ satisfies (NR1) and $t_2t_3$ satisfies (NR3). Then there exists only one situation, that is, $t_1 = \mathbf{s}(t_2)$  and  $t_2 \in T^{-1}$, $t_3 \in T$, $t_2^{-1} \neq  t_3$. Then $t_1t_2 \rightarrow t_2$ and $t_2t_3 \rightarrow 0$, and further we get $(t_1t_2)t_3 \rightarrow t_2t_3 \rightarrow 0$ and $t_1(t_2t_3)\rightarrow t_10 \rightarrow 0$.

Dually, if $t_1t_2$ satisfies (NR3) and $t_2t_3$ satisfies (NR1). Then there exists only one situation, that is, $t_1 \in T^{-1}$, $t_2 \in T$, $t_1^{-1} \neq t_2$ and $t_3 = \mathbf{r}(t_2)$. Then we still get $(t_1t_2)t_3 \rightarrow 0$ and $t_1(t_2t_3) \rightarrow 0$.

\vspace{2mm}
Case 4. $t_1t_2$ satisfies (NR1) and $t_2t_3$ satisfies (NR4). Then there exists only  one situation, that is, $t_1 = \mathbf{s}(t_2)$ and $t_2 = t_3^{-1}$ for $t_3 \in T$. Then $t_1t_2 \rightarrow t_2$ and $t_2t_3 \rightarrow \mathbf{r}(t_3)$, further we get $(t_1t_2)t_3 \rightarrow t_2t_3 \rightarrow \mathbf{r}(t_3)$ and $t_1(t_2t_3)\rightarrow t_1\mathbf{r}(t_3) \rightarrow \mathbf{r}(t_3)$ since $t_2 = t_3^{-1}$ and then $t_1= \mathbf{s}(t_2) = \mathbf{r}(t_3)$.

Dually, if $t_1t_2$ satisfies (NR4) and $t_2t_3$ satisfies (NR1). Then there exists only  one situation, that is, $t_1 \in T^{-1}$, $t_2 \in T$, $t_1^{-1} = t_2$ and $t_3 = \mathbf{r}(t_2)$. Then we get $(t_1t_2)t_3 \rightarrow \mathbf{r}(t_2)$ and
 $t_1(t_2t_3) \rightarrow \mathbf{r}(t_2)$.

\vspace{2mm}
 Case 5. $t_1t_2$ satisfies (NR1) and $t_2t_3$ satisfies (NR5). Then there exists only  one situation, that is, $t_1 = \mathbf{s}(t_2)$, $t_2 \in T^{-1}$, $t_3 = A \in T^0$ and $\mathbf{r}(t_2) \neq t_3$. Then we get $t_1t_2 \rightarrow t_2$ and $t_2t_3 \rightarrow 0$, further we get $(t_1t_2)t_3 \rightarrow t_2t_3 \rightarrow 0$  and $t_1(t_2t_3)\rightarrow t_10 \rightarrow 0$.

Dually, if $t_1t_2$ satisfies (NR5) and $t_2t_3$ satisfies (NR1). Then there exists two cases as follows:

(a1) When $t_1 \in T^{-1}$, $t_2 = A \in T^0$, $\mathbf{r}(t_1) \neq t_2$ and $t_2 = \mathbf{s}(t_3)$.  Then we get $t_1t_2 \rightarrow 0$ and $t_2t_3 = t_3$, further we get $(t_1t_2)t_3 \rightarrow 0t_3 \rightarrow 0$ and $t_1(t_2t_3) \rightarrow t_1t_3$. If $t_3 \in T^0$ then $t_2 = \mathbf{s}(t_3) = t_3$ and so $t_1t_3 \rightarrow 0$ by (NR5); else if $t_3 \in T$  then $t_3^{-1} \neq t_1$ as $\mathbf{r}(t_1) \neq t_2 = \mathbf{s}(t_3)$, and so $t_1t_3 \rightarrow 0$ by (NR3); otherwise $t_3 \in T^{-1}$, $t_1t_3 \rightarrow t_1\mathbf{s}(t_3) t_3 = (t_1\mathbf{s}(t_3)) t_3 \rightarrow 0t_3 = 0$ by (NR5).

(a2) When $t_1 \in T^{-1}$, $t_2 = A \in T^0$, $\mathbf{r}(t_1) \neq t_2$ and $t_3 = \mathbf{r}(t_2)$.  Then $t_2 = t_3 \in T^0$, and so $(t_1t_2)t_3 \rightarrow 0t_3 \rightarrow 0$ and $t_1(t_2t_3) \rightarrow t_1t_2 \rightarrow 0$.

\vspace{2mm}
Case 6. $t_1t_2$ satisfies (NR2) and $t_2t_3$ satisfies (NR2) or (NR3) or (NR5). Then we get $(t_1t_2)t_3 \rightarrow 0t_3 \rightarrow 0$ and $t_1(t_2t_3) \rightarrow t_10 \rightarrow 0$.

Dually, if $t_1t_2$ satisfies (NR2) or (NR3) or (NR5), and $t_2t_3$ satisfies (NR2). Then we get $(t_1t_2)t_3 \rightarrow 0t_3 \rightarrow 0$ and $t_1(t_2t_3) \rightarrow t_10 \rightarrow 0$.

Case 7. $t_1t_2$ satisfies (NR2) and $t_2t_3$ satisfies (NR4). Then we get $(t_1t_2)t_3 \rightarrow 0t_3 \rightarrow 0$ and  $t_1(t_2t_3) \rightarrow t_1\mathbf{r}(t_3) = t_1\mathbf{s}(t_2) \rightarrow 0$.

Dually, if $t_1t_2$ satisfies (NR4) and $t_2t_3$ satisfies (NR2). Then we get $(t_1t_2)t_3 \rightarrow \mathbf{r}(t_2)t_3 \rightarrow 0$ and  $t_1(t_2t_3) \rightarrow t_10 \rightarrow 0$.
\end{proof}

The reader maybe ask why we do not set $AB = A \cap B$. If we did so then for $t \in T$, $A \in T^0$ and ${\mathbf{r}}(t^{-1}) \subsetneq A$, then we have
\[t^{-1}{\mathbf{r}}(t^{-1})A = (t^{-1}{\mathbf{r}}(t^{-1}))A \rightarrow t^{-1}A \rightarrow 0\]
 by (NR1) and (NR5), and also we have
 \[t^{-1}{\mathbf{r}}(t^{-1})A = t^{-1}({\mathbf{r}}(t^{-1})A ) \rightarrow t^{-1}{\mathbf{r}}(t^{-1}) \rightarrow t^{-1}\]
 by ${\mathbf{r}}(t^{-1})A = {\mathbf{r}}(t^{-1}) \cap A = {\mathbf{r}}(t^{-1})$ and (NR1). Hence if we have $AB = A \cap B$ then the relation is not confluent.

 Certainly, the reduction system $(X^+, \rightarrow)$ where $X = T \cup T^0 \cup T^{-1}$ is a noetherian rewriting system. Consequently,  every element of $Q_{\Gamma}$ has a unique normal form as a word in $X^+$ by Lemma~\ref{rwriting}.

Next we describe the unique normal form for every element of $Q_{\Gamma}$.

Notice that if we start with a non-zero path $\alpha$ then any step in the reduction of $\alpha$ only used (NR1) and results in the non-zero path with the same source and range as $\alpha$. So for any non-zero path we can get a unique irreducible path. A path $\alpha = t_1t_2\ldots t_n$ is  {\it reduced} (or {\it irreducible} if
 \[t_{n-1} \neq \mathbf{s}(t_{n}) ,   t_{i-1} \neq \mathbf{s}(t_{i}),  \ \mbox{and}\ t_{i+1} \neq  \mathbf{r}(t_{i})\]
   for $1 < i < n$, where $t_1, t_2, \ldots, t_n \in T \cup T^0$ and $n \in {\mathbb{N}}$.

\begin{lemma}\label{ReducedPath}
For any non-zero path $\alpha \in P(\Gamma)$,  $\mathbf{s}(\alpha) = \mathbf{s}(\alpha')$ and $\mathbf{r}(\alpha) = \mathbf{r}(\alpha')$, where $\alpha'$ is the unique reduced path such that $[\alpha] = [\alpha']$.
\end{lemma}

 It is a good place to remark that the reduced path of a linear path is still a linear path.  $0$ is reduced.  We put
\[RP(\Gamma) = \{\alpha \in P({\Gamma}): \alpha \ \mbox{is\ reduced}\}\]
and
\[RLP(\Gamma) = \{\alpha \in LP({\Gamma}):  \alpha \ \mbox{is\ reduced}\}.\]
Notice that $T \cup T^0 \in RLP(\Gamma) \subset RP(\Gamma)$. In fact, we have
\[RLP(\Gamma) = \{t_1\cdots t_n: t_1, \ldots, t_n \in T  \ \mbox{and}\ \mathbf{r}(t_i)=\mathbf{s}(t_{i+1})\ \mbox{for}\ i = 1, \ldots, n-1\} \cup T^0.\]

The following lemma is clear so we omit its proof.
\begin{lemma}\label{RLP_Property}
For any $\alpha=t_1t_2\cdots{t_n}\in{RLP(\Gamma)}$  we have $t_it_{i+1}\cdots{t_j}\in{RLP(\Gamma)}$ for $1\leq{i} < j\leq{n}$.
 \end{lemma}

\begin{them}\label{w4}
Each element of $Q_{\Gamma}$ has a unique normal form of one of the following types:
\[{\rm (i)}\ [\alpha]; \  {\rm (ii)}\ [\beta^{-1}];\ {\rm (iii)}\ [\alpha\beta^{-1}] \ \mbox{and}\ {\rm (iv)}\ [0],\]
where $\alpha \in RP(\Gamma)$, $\beta \in RLP(\Gamma)$ and in {\rm (iii)} $\mathbf{r}(\alpha) \cap\mathbf{r}(\beta) \neq \emptyset$.
\end{them}
\begin{proof}
Let $[w] \in Q_{\Gamma}$. By Proposition~\ref{confluent} we can assume $w$ is reduced. If $w \neq 0$ and $w$ is reduced, $w$ does not contain a subword any $x^{-1}y$ where $x \in T$ and $y \in T \cup T^0$ nor any $x^{-1}y^{-1}$ where $x, y \in T$ and $yx$ is not linear. It follows that $w$ must be of the form $\alpha$, $\beta^{-1}$ or $\alpha\beta^{-1}$, where  $\alpha \in RP(\Gamma)$, $\beta \in RLP(\Gamma)$ and for $\alpha\beta^{-1}$,  $\mathbf{r}(\alpha) \cap\mathbf{r}(\beta) \neq \emptyset$.
\end{proof}

We will say that a word $w =\alpha\beta^{-1}$ with $\alpha \in RP(\Gamma)$, $\beta \in RLP(\Gamma)$ and $\mathbf{r}(\alpha) = \mathbf{r}(\beta)$  is a {\it right normal form}. In particular, $0$ is in right normal form.

Note that elements of type (i)-(iii) in Theorem~\ref{w4} may be uniquely expressed as $[\alpha\beta^{-1}]$ with $\mathbf{r}(\alpha) = \mathbf{r}(\beta)$, where in case (i) $[\alpha] = [\alpha\beta^{-1}]$ where $\beta = \mathbf{r}(\alpha)$, in case (ii) $[\beta^{-1}] = [\alpha\beta^{-1}]$ where $\alpha=\mathbf{r}(\beta)$, and in case (iii)
$$ [\alpha\beta^{-1}]=
	\begin{cases}
		[\alpha\beta^{-1}] & \mbox{if}\ \mathbf{r}(\alpha)=\mathbf{r}(\beta)\\
        [\mu\beta^{-1}] & \mbox{otherwise},
	\end{cases}$$
where $\mu=\alpha\mathbf{r}(\beta)$.

\begin{corollary}\label{uniqueNF}
Each non-zero element of  $Q_{\Gamma}$ has a unique right normal form representative $\alpha\beta^{-1}$, where $\alpha \in RP(\Gamma)$, $\beta \in RLP(\Gamma)$ and $\mathbf{r}(\alpha) = \mathbf{r}(\beta)$.
\end{corollary}

\begin{lemma}\label{aa-1}
If $\alpha \in {LP(\Gamma)}\setminus\{0\}$ then  $[\alpha^{-1}\alpha]=[\mathbf{r}(\alpha)]$.
\end{lemma}
\begin{proof}
 Suppose that $\alpha \in LP(\Gamma)\setminus\{0\}$.  Clearly, if $\alpha \in T \cup T^0$, then $[\alpha^{-1}\alpha]=[\mathbf{r}(\alpha)]$ by (NR1) and (NR4). If $\alpha=t_1t_2\cdots{t_n}$ with $\mathbf{r}(t_i)=\mathbf{s}(t_{i+1})$ for $i=1,2,\cdots,n-1$, then
	$$
	\left.
	\begin{array}{ll}
		[\alpha^{-1}\alpha] &= [t_n^{-1}\cdots{t_2^{-1}}t_1^{-1}t_1t_2\cdots{t_n}]\\
		&= [t_n^{-1}\cdots{t_2^{-1}}\mathbf{r}(t_1)t_2\cdots{t_n}] \hspace{18mm}\big(\mbox{by\ (NR4)}\big)\\
		&= [t_n^{-1}\cdots {t_3^{-1}}\mathbf{r}(t_2)t_3\cdots{t_n}]\hspace{18mm}\big( \mathbf{s}(t_2)=\mathbf{r}(t_1)\big)\\
		&= \cdots \\
		&= [t_n^{-1}\mathbf{r}(t_{n-1})t_n]\hspace{32mm}\big( \mathbf{s}(t_{n-1})=\mathbf{r}(t_{n-2})\big)\\
		&= [\mathbf{r}(t_{n})] \hspace{43mm}\big( \mathbf{s}(t_{n})=\mathbf{r}(t_{n-1})\big)\\
        &= [\mathbf{r}(\alpha)].
	\end{array}
	\right.
	$$
\end{proof}

Moreover, we have

\begin{lemma}\label{ProductS}

For any $[\alpha\beta^{-1}],[\mu\nu^{-1}]\in{Q_\Gamma}$ with $\alpha\beta^{-1}$ and $\mu\nu^{-1}$ in right normal form, we have

$$[\alpha\beta^{-1}][\mu\nu^{-1}] =
\begin{cases}
    {[\alpha\mu\nu^{-1}]} & \mbox{if}\ \beta = \mathbf{r}(\alpha)\ \mbox{and}\ \mathbf{r}(\alpha) \cap \mathbf{s}(\mu) \neq \emptyset\\
	{[\alpha\xi\nu^{-1}]} & \mbox{if}\  \beta \in RLP(\Gamma)\setminus T^0, \mu=\beta\xi\ \mbox{for\ some}\ \xi\in{RP(\Gamma)}\\
	{[\alpha(\nu\eta)^{-1}]} & \mbox{if}\ \beta \in RLP(\Gamma)\setminus T^0, \beta=\mu\eta\ \mbox{for\ some}\ \eta\in{RP(\Gamma)}\\
	[0] & \mbox{otherwise}.
\end{cases}$$
\end{lemma}
\begin{proof}
Suppose that $[\alpha\beta^{-1}],[\mu\nu^{-1}]\in{Q_\Gamma}$ with $\alpha\beta^{-1}$ and $\mu\nu^{-1}$ in right normal form. Then $\beta \in RLP(\Gamma)$.

If $\beta = \mathbf{r}(\alpha)$ then either $\mathbf{r}(\alpha) \cap \mathbf{r}(\mu) = \emptyset$ or $\mathbf{r}(\alpha) \cap \mathbf{r}(\mu) \neq \emptyset$. In the former, by (NR2) $[\alpha\beta^{-1}][\mu\nu^{-1}] = [0]$; in the latter, we get $[\alpha\beta^{-1}][\mu\nu^{-1}] = [\alpha\mathbf{r}(\alpha)\mu\nu^{-1}] = [\alpha\mu\nu^{-1}]$ by (NR1).

If $\beta \in RLP(\Gamma)\setminus T^0$, and $\mu=\beta\xi$ for some $\xi\in{RP(\Gamma)}$, then we have
$$[\alpha\beta^{-1}][\mu\nu^{-1}] = [(\alpha\beta^{-1})(\mu\nu^{-1})] = [(\alpha\beta^{-1})(\beta\xi\nu^{-1})] =[\alpha(\beta^{-1}\beta)\xi\nu^{-1}] = [\alpha \mathbf{r}(\beta)\xi\nu^{-1}] = [\alpha\xi\nu^{-1}] $$
 as $\mathbf{r}(\alpha) = \mathbf{r}(\beta)$ and (NR1).

If $\beta \in RLP(\Gamma)\setminus T^0$, and  $\beta=\mu\eta$ for some $\eta\in{RP(\Gamma)}$, then we have $\mu, \eta \in RLP(\Gamma)$, $\mathbf{r}(\mu) = \mathbf{s}(\eta)$ as $\beta \in RLP(\Gamma)$ and  Lemma~\ref{RLP_Property}, which follows from Lemma~\ref{aa-1} that $\mu^{-1}\mu = \mathbf{r}(\mu)$, and also $[\alpha\beta^{-1}][\mu\nu^{-1}] = [(\alpha\beta^{-1})(\mu\nu^{-1})] = [\alpha(\mu\eta)^{-1}\mu\nu^{-1}]=[\alpha\eta^{-1}\mu^{-1}\mu\nu^{-1}] = [\alpha\eta^{-1}\mathbf{r}(\mu)\nu^{-1}]= [\alpha\eta^{-1}\nu^{-1}] =[\alpha(\nu\eta)^{-1}]$ as $\mathbf{s}(
\eta)= \mathbf{r}(\mu) = \mathbf{r}(\nu)$ and (NR1).

If  $\beta=x_1x_2\cdots{x_m}$ and $\mu=y_1y_2\cdots{y_n}$ are not prefix comparable, then there exists $x_i\ne{y_i}$, for $1\leq{i}\leq{m,n}$. Suppose that $k$ is the smallest $i$ such that $x_k\ne{y_k}$ and $x_j=y_j$ for $j=1,2,\cdots,k-1$. From $\beta \in RLP(\Gamma)$ we obtain that $x_1\cdots x_{k-1} \in RLP(\Gamma)$ and $\mathbf{r}(x_{k-1}) = \mathbf{s}(x_k)$ by Lemma~\ref{RLP_Property}, and so $[x^{-1}_{k-1}\cdots x^{-1}_1x_1 \cdots x_{k-1}] = [\mathbf{r}(x_{k-1})]$ by Lemma~\ref{aa-1}. If $y_k \in T^0$ and $y_{k} = \mathbf{s}(x_k)$ we get $\mu = y_1y_2 \cdots y_{k-1}\mathbf{s}(x_k)y_{k+1}\cdots y_n$. Since $y_{k-1} = x_{k-1}$, $\mathbf{r}(x_{k-1}) = \mathbf{s}(x_k)$ we obtain that $\mathbf{r}(y_{k-1}) = \mathbf{s}(x_k)$ and so $y_{k-1}\mathbf{s}(x_k) = y_{k-1}$, which implies that $\mu$ is not reduced, a  contradiction. Thus $y_k \neq \mathbf{s}(x_k)$.  Further we can deduce that
     $$
     	 \left.
     	 \begin{array}{ll}
     	 	[\beta^{-1}\mu] &= [x^{-1}_m\cdots x^{-1}_{k+1}{x^{-1}_k}x^{-1}_{k-1}\cdots x^{-1}_1y_1\cdots{y_{k-1}}{y_k}\cdots{y_n}]\\
     	 	&= [x^{-1}_m\cdots{x^{-1}_{k+1}}{x^{-1}_k}x^{-1}_{k-1}\cdots x^{-1}_1x_1\cdots x_{k-1}{y_k}\cdots {y_n}]\\
      &\hspace{45mm}\big(x_j = y_j, j=1, 2, \ldots, k-1\big)\\
     	 	&= [x^{-1}_m\cdots{x^{-1}_{k+1}}{x^{-1}_k}\mathbf{r}(x_{k-1}){y_k}\cdots{y_n}]\hspace{6mm} \big(\mbox{by \ Lemma}\ \ref{aa-1}\big)\\
     	 	&= [x^{-1}_m\cdots{x^{-1}_{k+1}}{x^{-1}_k}y_k\cdots y_n]\\
      &\hspace{40mm}\big(\mathbf{r}(x_{k-1})=\mathbf{s}(x_k)\ \mbox{since}\ \beta \in \mbox{RLP}(\Gamma), \mbox{by\ (NR1)}\big)\\
     	 	&= [x^{-1}_m\cdots{x^{-1}_{k+1}}0y_{k+1}\cdots {y_n}] \hspace{7mm}\big(\mbox{either}\ x_k\neq y_k \in T, {\mbox{by\ (NR3)}} \\
     &\hspace{55mm} \mbox{or}\ y_k \in T^0, y_k \neq \mathbf{s}(x_k), {\mbox{by\ (NR5)}}\big)\\
      	    &= [0],
     	 \end{array}
     	 \right.
     	 $$
which follows that $[\alpha\beta^{-1}][\mu\nu^{-1}]=[\alpha(\beta^{-1}\mu)\nu^{-1}]=[0]$.
\end{proof}

\begin{lemma}\label{ES}
 The idempotent set of $Q_{\Gamma}$ is $E(Q_\Gamma)=\{[\alpha\alpha^{-1}]: \alpha\in{RLP(\Gamma)}\}\cup\{[0]\}$, which contains a subsemilattice $E= E(Q_{\Gamma})\setminus \{[A]: A \in T^0\}=\{[\alpha\alpha^{-1}]: \alpha \in RLP(\Gamma)\setminus T^0\} \cup \{0\}$.
\end{lemma}
\begin{proof}
It is clear that $[0] \in E(Q_\Gamma)$. It follows from Lemma~\ref{aa-1} that for all $\alpha \in RLP(\Gamma)$, $[\alpha\alpha^{-1}] \in E(Q_{\Gamma})$. Conversely,  Suppose that $[\alpha\beta^{-1}]\in{Q_\Gamma\setminus\{[0]\}}$ is such that $\alpha\beta^{-1}$ is in right normal form.  Assume that $[\alpha\beta^{-1}][\alpha\beta^{-1}]=[\alpha\beta^{-1}]$.   It follows from Lemma~\ref{ProductS} that either $\beta =\mathbf{r}(\alpha)$ and $\mathbf{r}(\alpha) \cap \mathbf{s}(\alpha) \neq \emptyset$ or $\beta \in RLP{\Gamma}\setminus T^0$ and $\alpha$ and $\beta$ are prefix  comparable.  In the former, we get $[\alpha\alpha]=[\alpha]$. It can only happen when $\alpha \in T^0$ as $[\alpha\beta^{-1}] \neq [0]$. It follows that $\alpha = \beta \in T^0 \subseteq RLP(\Gamma)$. In the latter, if  $\alpha = \beta\xi$ for some $\xi\in{RP(\Gamma)}$, then $[\alpha\beta^{-1}] =[\alpha\xi\beta^{-1}]$ which implies that $[\alpha] = [\alpha\xi]$ and then $\xi = \mathbf{r}(\alpha)$ as $\alpha \in RP(\Gamma)$. Thus $\alpha = \beta\xi = \beta \mathbf{r}(\alpha) = \beta \in RLP(\Gamma)$ as $\mathbf{r}(\alpha) = \mathbf{r}(\beta)$. Similarly, if $\beta= \alpha\eta$ for some $\eta\in{RP(\Gamma)}$, we get  that $\alpha = \beta \in RLP(\Gamma)$. Hence $E(Q_\Gamma)=\{[\alpha\alpha^{-1}]: \alpha\in{RLP(\Gamma)}\}\cup\{[0]\}$.

Next we show that $E=E(Q_{\Gamma})\setminus \{[A]: A \in T^0\}=\{[\alpha\alpha^{-1}]: \alpha \in RLP(\Gamma)\setminus T^0\} \cup \{0\}$ is a semilattice. Of course for all $[\alpha\alpha^{-1}] \in E$ we have $[0] [\alpha\alpha^{-1}] = [\alpha\alpha^{-1}][0]=[0]$ . Suppose that $[\alpha\alpha^{-1}],[\beta\beta^{-1}]\in E\setminus \{[0]\}$, we have
	$$
	\left.
	\begin{array}{ll}
		{[\alpha\alpha^{-1}][\beta\beta^{-1}]} &=
	      \begin{cases}
			{[\beta\beta^{-1}]} & \mbox{if}\ \beta=  \alpha\xi \ \mbox{for\ some}\ \xi\in{RP(\Gamma)}\\
			{[\alpha\alpha^{-1}]} & \mbox{if}\  \alpha = \beta\eta \ \mbox{for\ some}\ \eta\in{RP(\Gamma)} \\
			[0] & \mbox{otherwise}
		\end{cases}\\
&={[\beta\beta^{-1}][\alpha\alpha^{-1}]}\\
		& \in E(S_\Gamma).
	\end{array}
	\right.
	$$
\end{proof}

Note that for two distinct elements $A, B \in T^0$ we have $[A], [B] \in E(Q_{\Gamma})$ and  if $A \cap B \neq \emptyset$ then $[AB] \notin E(Q_{\Gamma})$ and also $[A][B] \neq [B][A]$. So in general, $E$ is not a band.

\begin{lemma}\label{regular}
For any $[\alpha\beta^{-1}]\in{Q_\Gamma}$ with $\alpha\beta^{-1}$ in right normal form, $[\alpha\beta^{-1}]$ is regular if and only if $\alpha\in{RLP(\Gamma)}$ or $[\alpha\beta^{-1}]=[0]$.
\end{lemma}
\begin{proof}
Clearly $[0]$ is regular.  If $\alpha\in{RLP(\Gamma)}$, then we have
$$
\left.
\begin{array}{ll}
	{[\alpha\beta^{-1}][\beta\alpha^{-1}][\alpha\beta^{-1}]}
	&= 	{[\alpha{\mathbf{r}(\beta)}(\alpha^{-1}\alpha)\beta^{-1}]} \\
	&={[\alpha\alpha^{-1}\alpha\beta^{-1}]}\\
	&={[ \alpha{\mathbf{r}(\alpha)}\beta^{-1}]}\\
	&={[\alpha\beta^{-1}]},
\end{array}
\right.
$$
which implies that $[\beta\alpha^{-1}]$ is an inverse of $[\alpha\beta^{-1}]$ and so $[\alpha\beta^{-1}]$ is regular.

Conversely, we assume that $[\alpha\beta^{-1}] \in Q_{\Gamma}\setminus \{[0]\}$ is a regular element, where $\alpha\beta^{-1}$ is in right normal form.  Assume that $[\mu\nu^{-1}]$ is an inverse of $[\alpha\beta^{-1}]$, where $\mu\nu^{-1}$ is in right normal form. Then $[\alpha\beta^{-1}][\mu\nu^{-1}] \in E(Q_{\Gamma})\setminus \{[0]\}$, and so there exists $\zeta \in RLP(\Gamma)$ such that $[\alpha\beta^{-1}][\mu\nu^{-1}] = [\zeta\zeta^{-1}]$. By Lemma~\ref{ProductS} $\alpha$ must be a prefix of $\zeta$. By Lemma~\ref{RLP_Property} we get $\alpha \in RLP(\Gamma)$.
\end{proof}

According to Lemma~\ref{regular} and its proof, if $[\alpha\beta^{-1}] \in Q_{\Gamma}\setminus \{[0]\}$ is regular, where $\alpha\beta^{-1}$ is in right normal form, then $\alpha, \beta \in RLP(\Gamma)$ and the inverse of $[\alpha\beta^{-1}]$ is $[\beta\alpha^{-1}]$. Further we get:
\begin{lemma}\label{regular LR}
If $[\alpha\beta^{-1}]\in{Q_\Gamma}\setminus \{[0]\}$ is regular with $\alpha\beta^{-1}$ in right normal form then $\alpha, \beta \in RLP(\Gamma)$ and $[\alpha\alpha^{-1}] \ {\mathcal{R}}\ [\alpha\beta^{-1}] \ {\mathcal{L}}\ [\beta\beta^{-1}]$, where $[\alpha\alpha^{-1}], [\beta\beta^{-1}] \in E(Q_{\Gamma})$.
\end{lemma}

\begin{lemma}\label{L*-class}
For any $[\alpha\beta^{-1}]\in{Q_\Gamma}\setminus \{[0]\}$ with $\alpha\beta^{-1}$ in right normal form, we have $[\alpha\beta^{-1}]\ {\mathcal{L^*}}\ [\beta\beta^{-1}]$.
\end{lemma}
\begin{proof}
For any $[\alpha\beta^{-1}]  \in{Q_\Gamma}\setminus \{[0]\}$ with $\alpha\beta^{-1}$ in right normal form, we get $\beta \in RLP(\Gamma)$ and so $[\beta\beta^{-1}] \in E(Q_{\Gamma})$ by Lemma~\ref{ES}. And also we have $[\alpha\beta^{-1}][\beta\beta^{-1}]=[\alpha\beta^{-1}]$. Suppose that for all $[x_1y_1^{-1}],[x_2y_2^{-1}]\in{Q_\Gamma}$  with  $x_1y_1^{-1}$ and $x_2y_2^{-1}$ in right normal form, $[\alpha\beta^{-1}][x_1y_1^{-1}]=[\alpha\beta^{-1}][x_2y_2^{-1}]$. Next  we show that $[\beta\beta^{-1}][x_1y_1^{-1}]=[\beta\beta^{-1}][x_2y_2^{-1}]$.

If $\beta =\mathbf{r}(\alpha)$, then  $[\beta\beta^{-1}] = [\mathbf{r}(\alpha)]$ and from $[\alpha\beta^{-1}][x_1y_1^{-1}]=[\alpha\beta^{-1}][x_2y_2^{-1}]$ we obtain that either $\mathbf{r}(\alpha) \cap \mathbf{s}(x_1) = \emptyset$ and $\mathbf{r}(\alpha) \cap \mathbf{s}(x_2) = \emptyset$ or $\mathbf{r}(\alpha) \cap \mathbf{s}(x_1) \neq \emptyset$ and $\mathbf{r}(\alpha) \cap \mathbf{s}(x_2) \neq \emptyset$. In the former, we certainly get $[\mathbf{r}(\alpha)][x_1y_1^{-1}]=[\mathbf{r}(\alpha)][x_2y_2^{-1}] =[0]$; in the latter, we get $[\alpha x_1y_1^{-1}]=[\alpha x_2y_2^{-1}]$ which implies that $[\mathbf{r}(\alpha) x_1y_1^{-1}]=[\mathbf{r}(\alpha)x_2y_2^{-1}]$. So $[\beta\beta^{-1}][x_1y_1^{-1}]=[\beta\beta^{-1}][x_2y_2^{-1}]$.

 Next assume that $\beta \neq \mathbf{r}(\alpha)$. According to the prefix comparable of $x_1$, $x_2$ with $\beta$, respectively, there exist five cases as follows:

{\bf Case 1.} $x_1=\beta\mu_1$ and $x_2=\beta\mu_2$ for some $\mu_1,\mu_2\in{RP(\Gamma)}$.

{\bf Case 2.} $x_1=\beta\mu_1$ and $\beta= x_2\mu_3$ for some $\mu_1,\mu_3\in{RP(\Gamma)}$.

{\bf Case 3.} $\beta= x_1\mu_4$ and $x_2=\beta\mu_2$ for some $\mu_2, \mu_4\in{RP(\Gamma)}$.

{\bf Case 4.}  $\beta= x_1\mu_4$ and $\beta= x_2\mu_3$ for some $\mu_3,\mu_4\in{RP(\Gamma)}$.

{\bf Case 5.} $\beta$ and $x_1$ are not prefix comparable and $\beta$ and $x_2$ are not prefix comparable.

Notice that in Case 5, we have $[\alpha\beta^{-1}][x_1y_1^{-1}]=[0]=[\alpha\beta^{-1}][x_2y_2^{-1}]$, and also $[\beta\beta^{-1}][x_1y_1^{-1}]=[0]=[\beta\beta^{-1}][x_2y_2^{-1}].$

It is sufficient to show Case 1 as Case 2, Case 3 and Case 4 are similar to Case 1. Now we consider Case 1. We have
 $$[\alpha\beta^{-1}][x_1y_1^{-1}]=[\alpha\mu_1y_1^{-1}] \neq [0]\ \mbox{and}\  [\alpha\beta^{-1}][x_2y_2^{-1}]=[\alpha\mu_2y_2^{-1}] \neq [0].$$
By the hypothesis we get $[\alpha\mu_1y_1^{-1}]=[\alpha\mu_2y_2^{-1}] \neq [0]$, which implies that $[\mathbf{r}(\alpha)\mu_1y^{-1}_1]=[\mathbf{r}(\alpha)\mu_2y^{-1}_2]$ as $\alpha \in RP(\Gamma)$, and then $[\beta][\mathbf{r}(\alpha)\mu_1y^{-1}_1]=[\beta][\mathbf{r}(\alpha)\mu_2y^{-1}_2]$, that is, $[\beta\mu_1y^{-1}_1]=[\beta\mu_2y^{-1}_2]$ as $\mathbf{r}(\alpha) = \mathbf{r}(\beta)$. Note that
$$[\beta\beta^{-1}][x_1y_1^{-1}]=[\beta\mu_1{y_1^{-1}}]\ and\ [\beta\beta^{-1}][x_2y_2^{-1}]=[\beta\mu_2y_2^{-1}].$$
It follows that $[\beta\beta^{-1}][x_1y_1^{-1}]=[\beta\beta^{-1}][x_2y_2^{-1}]$. Hence we have $[\alpha\beta^{-1}]\ {\mathcal{L}}^*\ [\beta\beta^{-1}]$ by Lemma~\ref{L*Def}.
\end{proof}

\begin{lemma}\label{LR*-class}
Suppose that $[\alpha\beta^{-1}]$, $[\mu\nu^{-1}]\in{Q_\Gamma} \setminus \{[0]\}$ where $\alpha\beta^{-1}$ and $\mu\nu^{-1}$ are in right normal form.
\begin{enumerate}
\item[\rm (i)]	$[\alpha\beta^{-1}]\ {\mathcal{L}}^*\ [\mu\nu^{-1}]$ if and only if $\beta = \nu$;
\item[\rm (ii)] if $\alpha,\mu\in{RLP(\Gamma)}$, then $[\alpha\beta^{-1}]\ {\mathcal{R}}\ [\mu\nu^{-1}]$ if and only if $\alpha = \mu$;
\item[\rm (iii)] if $\alpha\in{RLP(\Gamma)}$ and $\mu\in{RP(\Gamma)}\backslash RLP(\Gamma)$, then $([\alpha\beta^{-1}], [\mu\nu^{-1}])\not\in {\mathcal{R}}^*$.
 \end{enumerate}
\end{lemma}

\begin{proof}
(i) By Lemma~\ref{L*-class}, we get $[\alpha\beta^{-1}]\ {\mathcal{L}}^*\ [\beta\beta^{-1}]$ and $[\mu\nu^{-1}]\ {\mathcal{L}}^*\ [\nu\nu^{-1}]$. Then $[\alpha\beta^{-1}]\ {\mathcal{L}}^*\ [\mu\nu^{-1}]$ if and only if  $[\beta\beta^{-1}]\ {\mathcal{L}}^*\ [\nu\nu^{-1}]$.

If $\beta = \mathbf{r}(\alpha) \in T^0$ and $\nu = \mathbf{r}(\mu) \in T^0$ then $[\beta\beta^{-1}]=[\beta]$ and $[\nu\nu^{-1}] = [\nu]$, and so $[\beta\beta^{-1}]\ {\mathcal{L}}^*\ [\nu\nu^{-1}]$ if and only if $[\beta] \ {\mathcal{L}}\ [\nu]$, if and only if $[\beta][\nu]=[\beta]$ and $[\nu][\beta]=[\nu]$, which implies that $\mathbf{r}(\nu) = \mathbf{r}(\beta)$, that is, $\nu = \beta$ as $\nu, \beta \in T^0$. Conversely, if $\nu = \beta$ we certainly get $[\alpha\beta^{-1}]\ {\mathcal{L}}^*\ [\mu\nu^{-1}]$.

If $\beta, \nu \in RLP(\Gamma)\setminus T^0$, by Lemma~\ref{ES} $[\beta\beta^{-1}][\nu\nu^{-1}] = [\nu\nu^{-1}][\beta\beta^{-1}]$ and so $[\beta\beta^{-1}]\ {\mathcal{L}}^*\ [\nu\nu^{-1}]$ if and only if $[\beta\beta^{-1}] = [\nu\nu^{-1}]$, that is, $\beta = \nu$.

If one of $\beta$ and $\nu$ is in $T^0$ and the other is in $RLP{\Gamma}\setminus T^0$  then $[\beta\beta^{-1}]\ {\mathcal{L}}^*\ [\nu\nu^{-1}]$ never happen. Since if $\beta \in T^0$ and $\nu \in RLP\setminus T^0$ then $[\beta\beta^{-1}], [\nu\nu^{-1}] \in E(Q_{\Gamma})$ and $[\beta\beta^{-1}] \neq [\beta\beta][\nu\nu^{-1}]$, so $([\beta\beta^{-1}], [\nu\nu^{-1}]) \notin {\mathcal{L}}^*$.

(ii) For all $\alpha,\mu\in{RLP(\Gamma)}$, we have $[\alpha\beta^{-1}]$ and $[\mu\nu^{-1}]$ are regular by Lemma~\ref{regular}. It follows from  Lemma~\ref{regular LR} that $[\alpha\beta^{-1}]\ {\mathcal{R}}\ [\alpha\alpha^{-1}]$ and $[\mu\nu^{-1}]\ {\mathcal{R}}\ [\mu\mu^{-1}]$. So $[\alpha\beta^{-1}]\ {\mathcal{R}}\ [\mu\nu^{-1}]$ if and only if $[\alpha\alpha^{-1}]\ {\mathcal{R}}\ [\mu\mu^{-1}]$, if and only if $\alpha=\mu$ as the similar proof of part (i).

(iii) Suppose that $\alpha\in{RLP(\Gamma)}$ and $\mu\in{RP({\Gamma})}\backslash{RLP(\Gamma)}$. Let $[x_1y_1^{-1}]$, $[x_2y_2^{-1}] \in Q_{\Gamma}$ be such that
\[\alpha= y_1\xi_1= y_2\xi_2, \ x_1\xi_1 = x_2\xi_2,\ y_1\eta = \mu, \]
and $y_2$ and $\mu$ are not prefix comparable, where $x_1y_1^{-1}$ and $x_2y_2^{-1}$ are in right normal form, $\xi \in RLP(\Gamma)$, $\eta \in RP(\Gamma)\setminus RLP(\Gamma)$. Then
\[\ [x_1y_1^{-1}][\alpha\beta^{-1}]=[x_1\xi_1\beta^{-1}]=[x_2\xi_2\beta^{-1}]=[x_2y_2^{-1}][\alpha\beta^{-1}]\]
 but also we get
 \[\ [x_1y_1^{-1}][\mu\nu^{-1}] =[x_1\eta\nu^{-1}]\ne 0=[x_2y_2^{-1}][\mu\nu^{-1}].\]
  Hence $([\alpha\beta^{-1}],[\mu\nu^{-1}])\not\in{\mathcal{R}}^*$.
\end{proof}

It follows from Lemma~\ref{LR*-class} {\rm (iii)} that not all ${\mathcal{R}}^*$-classes of $Q_{\Gamma}$ have an idempotent, and so $Q_{\Gamma}$ is not left abundant. We also obtain that each ${\mathcal{L}}^*$-class of $Q_{\Gamma}$ has a unique idempotent by Lemma~\ref{LR*-class} {\rm (i)}. Then we have:

\begin{them}
The semigroup $Q_{\Gamma}$ is a right $*$-abundant semigroup with a zero.
\end{them}

Put
\[S_{\Gamma} = \{[\alpha\beta^{-1}] \in Q_{\Gamma}: \alpha \in PR(\Gamma),\ \beta \in RLP(\Gamma)\setminus T^0,\ \mathbf{r}(\alpha) = \mathbf{r}(\beta)\} \cup \{[0]\}.\]
Clearly, the subsemilattice
 $$E=E(Q_{\Gamma})\setminus \{[A]: A \in T^0\}=\{[\alpha\alpha^{-1}]: \alpha \in RLP(\Gamma)\setminus T^0\} \cup \{0\}$$
 is the set of all idempotents of $S_{\Gamma}$. It is easy to see that $S_{\Gamma}$ forms a subsemigroup of $Q_{\Gamma}$ by Lemma~\ref{ProductS}. Since for all $[\alpha\beta^{-1}] \in S_{\Gamma}$, $\beta \in RLP(\Gamma)\setminus T^0$, the product in $S_{\Gamma}$ can be modified as follows: for all $[\alpha\beta^{-1}], [\mu\nu^{-1}] \in S_{\Gamma}$,
$$[\alpha\beta^{-1}][\mu\nu^{-1}] =
\begin{cases}
    {[\alpha\xi\nu^{-1}]} & \mbox{if}\  \beta \in RLP(\Gamma)\setminus T^0, \mu=\beta\xi\ \mbox{for\ some}\ \xi\in{RP(\Gamma)}\\
	{[\alpha(\nu\eta)^{-1}]} & \mbox{if}\ \beta \in RLP(\Gamma)\setminus T^0, \beta=\mu\eta\ \mbox{for\ some}\ \eta\in{RP(\Gamma)}\\
	[0] & \mbox{otherwise}.
\end{cases}$$

\begin{prop}\label{SG}
The semigroup $S_\Gamma$ is a right ample semigroup with a zero element.
\end{prop}
\begin{proof}
According to Lemma~\ref{ES} and Lemma~\ref{L*-class} it is sufficient to show that for all $[\xi\xi^{-1}]\in{E}$ and $[\alpha\beta^{-1}]\in{S_\Gamma} \setminus \{[0]\}$,  $[\xi\xi^{-1}][\alpha\beta^{-1}]=[\alpha\beta^{-1}]([\xi\xi^{-1}][\alpha\beta^{-1}])^*$. Notice that
$$
\left.
\begin{array}{ll}
{[\xi\xi^{-1}][\alpha\beta^{-1}]}
		&= \begin{cases}
		{[\alpha\beta^{-1}]} & \mbox{if}\ \alpha=\xi\mu \ \mbox{for\ some}\ \xi\in{RP(\Gamma)}\\
		{[\xi(\beta\eta)^{-1}]} & \mbox{if}\ \xi=\alpha\eta \ \mbox{for\ some}\ \eta\in{RP(\Gamma)}\\
		[0] & \mbox{otherwise}.
	\end{cases}\\
\end{array}
\right.
$$
Then
$$([\xi\xi^{-1}][\alpha\beta^{-1}])^* =
\begin{cases}
	{[\beta\beta^{-1} ]}& \mbox{if}\ \alpha=\xi\mu \ \mbox{for\ some}\ \xi\in{RP(\Gamma)}\\
	{[(\beta\eta)(\beta\eta)^{-1}]} & \mbox{if}\ \xi=\alpha\eta \ \mbox{for\ some}\ \eta\in{RP(\Gamma)}\\
	[0] & \mbox{otherwise}.
\end{cases}$$
And so
$$
\left.
\begin{array}{ll}
	&{[\alpha\beta^{-1}]([\xi\xi^{-1}][\alpha\beta^{-1}])^*}\\
	&= \begin{cases}
		{[\alpha\beta^{-1}][\beta\beta^{-1}]} & \mbox{if}\ \alpha=\xi\mu \ \mbox{for\ some}\ \xi\in{RP(\Gamma)}\\
		{[\alpha\beta^{-1}][\beta\eta(\beta\eta)^{-1}]} & \mbox{if}\ \xi=\alpha\eta \ \mbox{for\ some}\ \eta\in{RP(\Gamma)}\\
		[0] & \mbox{otherwise}
	\end{cases}\\
	&= \begin{cases}
		{[\alpha\beta^{-1}]} & \mbox{if}\ \alpha=\xi\mu \ \mbox{for\ some}\ \xi\in{RP(\Gamma)}\\
		{[\xi(\beta\eta)^{-1}]} & \mbox{if}\ \xi=\alpha\eta \ \mbox{for\ some}\ \eta\in{RP(\Gamma)}\\
		[0] & \mbox{otherwise}.
	\end{cases}\\
\end{array}
\right.
$$
Thus $[\xi\xi^{-1}][\alpha\beta^{-1}]=[\alpha\beta^{-1}]([\xi\xi^{-1}][\alpha\beta^{-1}])^*$. Hence $S_{\Gamma}$ is a right ample semigroup with a zero element.
\end{proof}

According to Lemma~\ref{regular} and Lemma~\ref{regular LR} we obtain that if $\alpha, \beta \in RLP(\Gamma)$ and $\mathbf{r}(\alpha) = \mathbf{r}(\beta)$ then $[\alpha\beta^{-1}]$ is regular. Let $R_{\Gamma}$ be the subset of  $S_{\Gamma}$ as follows:
\[R_{\Gamma} = \{[\alpha\beta^{-1}] \in S_{\Gamma}: \alpha, \beta \in RLP(\Gamma)\setminus T^0 \mbox{and}\ \mathbf{r}(\alpha) =\mathbf{r}(\beta)\} \cup \{[0]\}.\] Then $E \subseteq R_{\Gamma}$. It follows from Lemma~\ref{inverse} that $R_{\Gamma}$ is an inverse subsemigroup of $S_{\Gamma}$. By Lemma~\ref{LR*-class} we obtain that $\mathcal{H}$ is trivial in $R_{\Gamma}$, where ${\mathcal{H}}={\mathcal{L}}\cap {\mathcal{R}}$.  A semigroup is fundamental if and only if
the only congruence contained in Green's equivalence $\mathcal{H}$ is the identity congruence~\cite{Munn}.

\begin{corollary}\label{R_Gamma}
The semigroup $R_{\Gamma}$ is a fundamental inverse subsemigroup  of $S_{\Gamma}$.
\end{corollary}

\section{Congruence-free}\label{sec:congruence}

The aim of this section is to consider the congruence-free conditions  on $Q_\Gamma$.

We begin with the notion of congruence-free.  A semigroup is said to be {\it congruence-free} if it has only two congruences, the identity congruence and the universal congruence.

\begin{lemma}\label{idempotent-separating}
Let $I= \{[\alpha\beta^{-1}]\in Q_{\Gamma}: \alpha \in RP(\Gamma)\setminus RLP(\Gamma)\} \cup \{[0]\}$. Then $I$ is a proper ideal of $Q_{\Gamma}$. Further $\rho_I = (I \times I) \cup 1_{Q_{\Gamma}}$ is an idempotent-separating congruence on $Q_{\Gamma}$.
\end{lemma}
\begin{proof}
It is immediate to get that $I$ is closed with respect to the binary operation in $Q_{\Gamma}$ by Lemma~\ref{ProductS} and $Q_{\Gamma}\setminus I \neq \emptyset$. Suppose that $[\alpha\beta^{-1}] \in I$ and  $[\mu\nu^{-1}] \in Q_{\Gamma}\setminus I$. Since $[\alpha\beta^{-1}] \in I$ it follows that $\alpha \in RP(\Gamma)\setminus RLP(\Gamma)$. Again by Lemma~\ref{ProductS} we have that $[\alpha\beta^{-1}][\mu\nu^{-1}] \in I$. Now we show that $[\mu\nu^{-1}][\alpha\beta^{-1}] \in I$.  As $\alpha \in RP(\Gamma)\setminus RLP(\Gamma)$ and  $\nu \in RLP(\Gamma)$  we can not have $ \nu=\alpha\eta$ for some $\eta \in RP(\Gamma)$.  Since if  $ \nu=\alpha\eta$ we get $\alpha \in RLP(\Gamma)$ by Lemma~\ref{RLP_Property}, a contradiction.  It follows from Lemma~\ref{ProductS} that
$$[\mu\nu^{-1}][\alpha\beta^{-1}]=
\begin{cases}
    [\mu\alpha\beta^{-1}] & \mbox{if}\ \nu = \mathbf{r}(\mu)\ \mbox{and}\ \mathbf{r}(\mu) \cap \mathbf{s}(\alpha) \neq \emptyset\\
	{[\mu\xi\nu^{-1}]} & \mbox{if}\  \alpha=\nu\xi\ \mbox{for\ some}\ \xi\in{RP(\Gamma)}\\
	[0] & \mbox{otherwise}.
\end{cases}$$

If $\nu = \mathbf{r}(\mu)$ and $\mathbf{r}(\mu) \cap \mathbf{s}(\alpha) \neq \emptyset$ then $\mu\alpha \in RP(\Gamma)\setminus RLP(\Gamma)$ since $\alpha \in RP(\Gamma)\setminus RLP(\Gamma)$ and so $[\mu\alpha\beta^{-1}] \in I$.  If $\alpha=\nu\xi$ for some $\xi\in{RP(\Gamma)}$ then either ${\mathbf{r}}(\nu) \neq {\mathbf{s}}(\xi)$ and  ${\mathbf{r}}(\nu) \cap {\mathbf{s}}(\xi)\neq \emptyset$ or $\mathbf{r}(\nu) = \mathbf{s}(\xi)$  and $\xi \in RP(\Gamma)\setminus RLP(\Gamma)$. In the former, we get $\mathbf{r}(\mu) \neq \mathbf{s}(\xi)$ and $\mathbf{r}(\mu) \cap \mathbf{s}(\xi)\neq \emptyset$ as $\mathbf{r}(\mu) = \mathbf{r}(\nu)$. And so $\mu\xi \in RP(\Gamma)\setminus RLP(\Gamma)$, which follows that $[\mu\xi\nu^{-1}] \in I$. In the latter, we certainly get $\mu\xi \in RP(\Gamma)\setminus RLP(\Gamma)$, which follows that $[\mu\xi\nu^{-1}] \in I$. Notice that $[0] \in I$. Hence $I$ is a proper ideal of $Q_{\Gamma}$.

It follows from Lemma~\ref{ES} that $I$ contains only one idempotent, that is, $[0]$. Consequently, $\rho_I$ is an idempotent-separating congruence on $Q_{\Gamma}$.
\end{proof}

We also can construct a proper ideal of $S_{\Gamma}$ as follows. Its proof is similar to Lemma~\ref{idempotent-separating} so we omit it.

\begin{lemma}\label{idempotent-separating2}
Let $I= \{[\alpha\beta^{-1}]\in S_{\Gamma}: \alpha \in RP(\Gamma)\setminus RLP(\Gamma)\} \cup \{[0]\}$. Then $I$ is a proper ideal of $S_{\Gamma}$. Further $\rho_I = (I \times I) \cup 1_{S_{\Gamma}}$ is an idempotent-separating congruence on $S_{\Gamma}$.
\end{lemma}

According to Lemma~\ref{regular}, Lemma~\ref{idempotent-separating} and Lemma~\ref{idempotent-separating2} $Q_{\Gamma}$ and $S_{\Gamma}$ are never  congruence-free if they have non-regular elements. Notice that if  $t$ is a relation  in $\Gamma$ with $\mid\mathbf{s}(t)\mid >1$ and $v \in \mathbf{s}(t)$, then we get $vt \in RP(\Gamma)\setminus RLP(\Gamma)$ and so $[vtt^{-1}] \in S_{\Gamma} \subseteq Q_{\Gamma}$ is a non-regular element. Now we have:

 \begin{them}
 If $\Gamma = (V, T, \mathbf{s}, \mathbf{r})$ is a  network and there exists $t \in T$ with $\mid\mathbf{s}(t)\mid >1$, then $Q_{\Gamma}$ and $S_{\Gamma}$ are not  congruence-free as a semigroup.
 \end{them}

Let $\Gamma = (V, T, \mathbf{r}, \mathbf{s})$ be a network. For each $A \subset V$ the cardinality of the set $\{t \in T: \mathbf{s}(t)=A\}$ is called the {\it out-index} of $A$ in $\Gamma$, denoted by $o(A)$.

A {\it $*$-ideal} of a right abundant semigroup $S$ is an ideal of $S$ which is closed under the relation ${\mathcal{L}}^{*}$. It is easy to see that if we regard a right $*$-abundant semigroup $S$ as a unary semigroup and $I$ is a proper $*$-ideal of $S$  then  $\rho_I = (I \times I) \cup 1_{S}$ is a unary semigroup congruence on $S$. Since if $a\, \rho_I\, b$ then either $a = b$ (so certainly $a^* = b^*$) or $a, b \in I$. But then $a^*, b^* \in I$ (as $I$ is closed under $^*$), so $a^*\, \rho_I\, b^*$.

 A right $*$-abundant semigroup is said to be {\it $*$-congruence-free} if it has only two unary semigroup congruences, the identity congruence and the universal congruence.

\begin{lemma}\label{*-ideal}
Let $\Gamma = (V, T, \mathbf{r}, \mathbf{s})$ be a network and let $t \in T$ be such that $o(\mathbf{r}(t))=0$ and there does not exist $A \in T^0\setminus V$ with $\mathbf{r}(t) \subseteq A$. Then the principal ideal $I$ generated by $[tt^{-1}]$ is a proper $*$-ideal of $Q_{\Gamma}$.
\end{lemma}
 \begin{proof}
 Let $t$ be a relation in a network $\Gamma$ such that $o(\mathbf{r}(t))=0$ and there does not exist $A \in T^0\setminus V$ with $\mathbf{r}(t) \subseteq A$. Then we obtain that $|\mathbf{r}(t)|=1$ otherwise $\mathbf{r}(t) \in T^0\setminus V$ and $\mathbf{r}(t) \subseteq \mathbf{r}(t)$, a contradiction. If $\beta \in RP(\Gamma)$ and $\beta = t\eta$ for some $\eta \in PR(\Gamma)$, we must have $\beta = t$ and $\eta = \mathbf{r}(t)$ as $o(\mathbf{r}(t))=0$. If $\beta \in RP(\Gamma)$ and $t = \beta\xi$ for some $\xi \in PR(\Gamma)$,  we have either $\beta = \mathbf{s}(t)$ and $\xi = t$ or $\beta = t$ and $\xi = \mathbf{r}(t)$.  So by Lemma~\ref{ProductS} we get the principal ideal $I$ generated by $[tt^{-1}]$  as follows
 \[I = Q_{\Gamma}[tt^{-1}]Q_{\Gamma}=\{[\alpha\beta^{-1}]: \alpha \in RP(\Gamma), \beta \in RLP(\Gamma), \mathbf{r}(\alpha)=\mathbf{r}(\beta)=\mathbf{r}(t)\} \cup \{[0]\}.\]
 It is easy to see that $I$ is closed under the relation ${\mathcal{L}}^{*}$ by Lemma~\ref{LR*-class}. Certainly $I$ is a proper ideal of $Q_{\Gamma}$ since $\mathbf{s}(t) \cap \mathbf{r}(t) = \emptyset$ and $[\mathbf{s}(t)] \notin I$.
 \end{proof}

 The following theorem is an immediate result of Lemma~\ref{*-ideal}.

 \begin{them}
If $\Gamma = (V, T, \mathbf{s}, \mathbf{r})$ is a  network and there exists $t \in T$ such that  $o(\mathbf{r}(t))=0$ and there does not exist $A \in T^0\setminus V$ with $\mathbf{r}(t) \subseteq A$, then $Q_{\Gamma}$ is not  $*$-congruence-free as a  unary semigroup.
\end{them}

 \begin{lemma}\label{*-ideal S}
Let $\Gamma = (V, T, \mathbf{r}, \mathbf{s})$ be a network  with $|T|> 1$ and let $t, q \in T$ be such that $o(\mathbf{r}(t))=0$, $\mathbf{r}(t) \neq \mathbf{r}(q)$  and there does not exist $A \in T^0\setminus V$ with $\mathbf{r}(t) \subseteq A$. Then the principal ideal  generated by $[tt^{-1}]$ is a proper $*$-ideal of $S_{\Gamma}$.
\end{lemma}
 \begin{proof}
 Let $t$ and $q$ be two distinct relations in a network $\Gamma$ such that $o(\mathbf{r}(t))=0$, $\mathbf{r}(t) \neq \mathbf{r}(q)$ and there does not exist $A \in T^0\setminus V$ with $\mathbf{r}(t) \subseteq A$. Similar to the proof of Lemma~\ref{*-ideal}, we have if $\beta \in RP(\Gamma)$ and $\beta = t\eta$ for some $\eta \in PR(\Gamma)$, then $\beta = t$ and $\eta = \mathbf{r}(t)$; if $\beta \in RP(\Gamma)$ and $t = \beta\xi$ for some $\xi \in PR(\Gamma)$, then either $\beta = \mathbf{s}(t)$ and $\xi = t$ or $\beta = t$ and $\xi = \mathbf{r}(t)$.   Again by the statement above Theorem~\ref{SG}  we get the principal ideal $I$ of $S_{\Gamma}$ generated by $[tt^{-1}]$ as follows
 \[I = S^1_{\Gamma}[tt^{-1}]S^1_{\Gamma}=\{[\alpha\beta^{-1}]: \alpha \in RP(\Gamma), \beta \in RLP(\Gamma)\setminus T^0, \mathbf{r}(\alpha)=\mathbf{r}(\beta)=\mathbf{r}(t)\} \cup \{[0]\}.\]
 It is easy to see that $I$ is closed under the relation ${\mathcal{L}}^{*}$ by Lemma~\ref{LR*-class}. Certainly $I$ is a proper ideal of $S_{\Gamma}$ since  $[qq^{-1}] \notin I$.
 \end{proof}

Similar to Lemma~\ref{*-ideal S} we construct a proper ideal of $R_{\Gamma}$ as follows. Its proof is the same as Lemma~\ref{*-ideal S} so we omit it.
\begin{lemma}\label{*-ideal R}
Let $\Gamma = (V, T, \mathbf{r}, \mathbf{s})$ be a network  with $|T|> 1$ and let $t, q \in T$ be such that $o(\mathbf{r}(t))=0$, $\mathbf{r}(t) \neq \mathbf{r}(q)$  and there does not exist $A \in T^0\setminus V$ with $\mathbf{r}(t) \subseteq A$. Then the principal ideal $R_{\Gamma}[tt^{-1}]R_{\Gamma}$  generated by $[tt^{-1}]$ is a proper ideal of $R_{\Gamma}$, where
\[ R_{\Gamma}[tt^{-1}]R_{\Gamma}=\{[\alpha\beta^{-1}]: \alpha, \beta \in RLP(\Gamma)\setminus T^0, \mathbf{r}(\alpha)=\mathbf{r}(\beta)=\mathbf{r}(t)\} \cup \{[0]\}.\]
\end{lemma}

 Further, we have:

  \begin{them}
If $\Gamma = (V, T, \mathbf{r}, \mathbf{s})$ is a network  with $|T|>1$ and let $t, q \in T$ be such that $o(\mathbf{r}(t))=0$, $\mathbf{r}(t) \neq \mathbf{r}(q)$   and there does not exist $A \in T^0\setminus V$ with $\mathbf{r}(t) \subseteq A$, then $S_{\Gamma}$ is not  $*$-congruence-free as a unary semigroup, and also $R_{\Gamma}$ is not  congruence-free.
\end{them}

\section{Homomorphisms}\label{sec:hom}
In this section we discuss the relationship between network homomorphisms and semigroup homomorphisms.

Let $S$ be an semigroup with a set $E(S)$ of all idempotents. The relation $\leq$ defined by for all $a, b \in S$,
\[a \leq b \ \mbox{if\ and\ only\ if}\ a=xb = by, \ xa = a\]
for some $x, y \in S^1$, is a partial order on $S$ called the {\it natural partial order} of $S$~\cite{Mitsch}. When the natural partial order is restricted to the set $E(S)$ it is as follows: for all $e, f \in E(S)$,
\[e \leq f\ \mbox{if\ and\ only\ if}\ e = ef = fe.\]
Further, $E(S)$ is a partially ordered set with respect to $\leq$. In particular, if $E(S)$ is a semilattice, for all $e, f \in E$,
\[e \leq f\ \mbox{if\ and\ only\ if}\ e = ef.\]

\begin{lemma}\label{ES_Partial}
Let $E(Q_\Gamma)$ be the set of all idempotents of $Q_\Gamma$ and let $\leq$ be the natural partial order on $Q_\Gamma$  defined above. Then the following statements hold.
\begin{enumerate}
\item[\rm (i)]  An idempotent $[\alpha\alpha^{-1}]$ is maximal in $E(Q_\Gamma)$ with respect to $\leq$ if and only if $\alpha \in T^0$;
\item[\rm (ii)] An idempotent $[\alpha\alpha^{-1}]$ is maximal  in $E=E(Q_\Gamma)\setminus{\{[A]: A \in T^0\}}$ with respect to $\leq$ if and only if $\alpha \in T$.
\end{enumerate}
\end{lemma}
\begin{proof}

(i) Suppose that $\alpha \in T^0$ and $[\alpha\alpha^{-1}]\leq[\mu\mu^{-1}]$ for some $\mu \in RLP(\Gamma)$. Then $[\mu\mu^{-1}]\ne [0]$ and $[\alpha\alpha^{-1}]=[\alpha\alpha^{-1}][\mu\mu^{-1}]=[\mu\mu^{-1}][\alpha\alpha^{-1}]$, that is, $[\alpha]=[\alpha\mu\mu^{-1}]=[\mu\mu^{-1}\alpha]$ which never happen if $\mu \in RLP(\Gamma)\setminus T^0$.  If $\mu \in T^0$, then it follows  from $[\alpha]=[\alpha\mu\mu^{-1}]=[\mu\mu^{-1}\alpha]$ that $[\alpha]=[\alpha\mu]=[\mu\alpha]$ which implies that $\alpha = \mu$. Hence $[\alpha\alpha^{-1}]$ is maximal in $E(Q_\Gamma)$.

Conversely, suppose that $\alpha \in RLP(\Gamma)$ and $[\alpha\alpha^{-1}]$ is maximal in $E(Q_\Gamma)$. Then $[\alpha\alpha^{-1}]\ne [0]$. Thus $[\alpha\alpha^{-1}] \leq [\mathbf{s}(\alpha)]$ since $[\alpha\alpha^{-1}]=[\mathbf{s}(\alpha)\alpha\alpha^{-1}]=[\alpha\alpha^{-1}\mathbf{s}(\alpha)]$. As $[\alpha\alpha^{-1}]$ is maximal, we get that $[\alpha\alpha^{-1}]=[\mathbf{s}(\alpha)]$, which implies that $\alpha = \mathbf{s}(\alpha) \in T^0$.

(ii) It follows from Lemma~\ref{ES} that $E=E(Q_\Gamma)\setminus{\{[A]: A \in T^0\}}$ is a semilattice. Let $\alpha\in T$, $\mu \in RLP(\Gamma)\setminus T^0$ and suppose that $[\alpha\alpha^{-1}] \leq [\mu\mu^{-1}]$ in $E$.
 Then $[\mu\mu^{-1}] \ne [0]$ and $[\alpha\alpha^{-1}][\mu\mu^{-1}]=[\alpha\alpha^{-1}]$. By the proof of Lemma~\ref{ES} we get $\alpha=\mu{x}$ for some $x \in RLP(\Gamma)$. Since $\alpha \in T$ and $\mu \in RLP(\Gamma)\setminus T^0$ it follows that  $\alpha=\mu$ and $x=\mathbf{r}(\alpha)$. Hence $[\alpha\alpha^{-1}]$ is maximal in $E(S_\Gamma)$.
	
Conversely, suppose that $[\alpha\alpha^{-1}]$ is maximal in $E$. Then $[\alpha\alpha^{-1}]\ne [0]$. Since $\alpha \in RLP(\Gamma)\setminus T^0$, we can write $\alpha= t\beta$ for some $t \in T$ and $\beta \in {RLP(\Gamma)}$, and hence $[\alpha\alpha^{-1}]=[t\beta(t\beta)^{-1}]\leq[{tt^{-1}}]$ by Lemma~\ref{ES}. Since $[\alpha\alpha^{-1}]$ is maximal in $E$, and $[tt^{-1}]\in E$, this implies that $[\alpha\alpha^{-1}]=[tt^{-1}]$. As $\alpha\alpha^{-1}$ and $tt^{-1}$ are in right normal form we get that $\alpha= t \in T$.
\end{proof}

\begin{them}\label{NetworkIsomorphism}
 Let $\Gamma = (V_{\Gamma}, T_{\Gamma}, \mathbf{s}, \mathbf{r})$ and $\Delta=(V_{\Delta}, T_{\Delta}, \mathbf{s}, \mathbf{r})$ be two networks. Then $\Gamma\cong{\Delta}$ if and only if $Q_\Gamma\cong{Q_\Delta}$.
\end{them}

\begin{proof}
  Let $\Gamma = (V_{\Gamma}, T_{\Gamma}, \mathbf{s}, \mathbf{r})$ and $\Delta=(V_{\Delta}, T_{\Delta}, \mathbf{s}, \mathbf{r})$  be two networks. It is sufficient to show that if $Q_{\Gamma}\cong{Q_\Delta}$ then $\Gamma\cong{\Delta}$. Let $\theta$ be a semigroup isomorphism from $Q_\Gamma$ to $Q_\Delta$. Let $E(Q_\Gamma)$ and $E(Q_\Delta)$ denote the subsets of all idempotents of $Q_\Gamma$ and $Q_\Delta$, respectively. Also let $\leq_{Q_\Gamma}$ and $\leq_{Q_\Delta}$ denote the restriction to $E(Q_\Gamma)$ and $E(Q_\Delta)$, respectively, of the natural partial order on $Q_\Gamma$ and $Q_\Delta$, respectively. Then $(E(Q_\Gamma),\leq_{Q_\Gamma})$ and $(E(Q_\Delta),\leq_{Q_\Delta})$ are partially ordered sets. Since any isomorphism of semigroups induces an order-isomorphism of the corresponding  partially ordered sets of idempotents, $\theta$ must take $E(Q_\Gamma)$ bijectively to $E(Q_\Delta)$. By part (i) of Lemma~\ref{ES_Partial}, every vertex corresponding $[v]$ in $E(Q_\Gamma)$, but no other element of $Q_\Gamma$, is maximal in $E(Q_\Gamma)$, with respect to $\leq_{Q_\Gamma}$, and analogously for $Q_\Delta$. It follows that the restriction of $\theta$ to $\{[v]: v \in V_\Gamma\}$ is a bijection from  $\{[v]: v \in V_\Gamma\}$ to $\{[v']: v' \in V_\Delta\}$.
	
Next, by Lemma~\ref{ES_Partial}, every element of the form $[tt^{-1}](t\in{T_\Gamma})$, but no other element of $Q_\Gamma$, is maximal in $E(Q_\Gamma)\setminus{\{[A]: A \in T^0_\Gamma\}}$ with respect to $\leq_{Q_\Gamma}$, and analogously for $Q_\Delta$. 	Hence $\theta$ must take $\{[tt^{-1}]: t\in{T_\Gamma}\}$ bijectively to $\{[qq^{-1}]: q\in{T_\Delta}\}$. Now, let $t\in{T_\Gamma}$, write $[tt^{-1}]\theta=[qq^{-1}]$ for some $q\in{T_\Delta}$, and write $[t]\theta=[xy^{-1}]$ for some $x\in{RP(\Delta)}$, $y\in{RLP(\Delta)}$ and $\mathbf{r}(x) =\mathbf{r}(y)$. Then
$[t^{-1}]\theta = [(xy^{-1})^{-1}] = [yx^{-1}] \neq [0]$, which implies that $x \in RLP(\Gamma)$ by Theorem~\ref{w4} and so
	$$[qq^{-1}]=[tt^{-1}]\theta=[t]\theta[t^{-1}]\theta=[xy^{-1}][yx^{-1}]=[xx^{-1}]$$
since $\theta$ is an isomorphism of  right abundant semigroups. It follows that $q= x$. Furthermore,
	$$[\mathbf{r}(t)]\theta=[t^{-1}t]\theta=[yx^{-1}xy^{-1}]=[yq^{-1}qy^{-1}]=[yy^{-1}],$$
which implies that $y\in{T^0_\Delta}$, that is, $y = \mathbf{r}(y)$. Therefore
\[[t]\theta=[xy^{-1}]= [x\mathbf{r}(y)]= [x\mathbf{r}(x)]= [x]=[q]\]
 as $\mathbf{r}(x)=\mathbf{r}(y)$ and $x = q$, and hence
 \[[\mathbf{r}(t)]\theta = [yy^{-1}]=[\mathbf{r}(y)\mathbf{r}(y)^{-1}] =[\mathbf{r}(y)]= [\mathbf{r}(x)] = [\mathbf{r}(q)]= [\mathbf{r}([t]\theta)].\]

 Certainly we have $[\mathbf{r}(t)]\theta \subseteq \{[v]\theta: v \in \mathbf{r}(t)\}$. Conversely, for any $v \in \mathbf{r}(t)$, we have $[tv] \neq [0]$ and so as $\theta$ is a homomorphism, we have $[tv]\theta \neq [0]$. Then $[v]\theta \in \mathbf{r}(t)\theta$. Consequently, we have
 \[[\mathbf{r}(t)]\theta = \{[v]\theta: v \in \mathbf{r}(t)\} =  [\mathbf{r}([t]\theta)].\]
  Moreover we have $[\mathbf{s}(t)]\theta=[\mathbf{s}(x)]=[\mathbf{s}(q)]=[\mathbf{s}([t]\theta)] = \{[v]\theta: v \in \mathbf{s}(t)\} $.

 Since $\theta$ takes
$\{[tt^{-1}]: t \in {T_\Gamma}\}$ bijectively to $\{[qq^{-1}]: q \in {T_\Delta}\}$, it follows that $\theta$ takes $T_\Gamma$ bijectively to $T_\Delta$. Choose $\theta_{[V_{\Gamma}]}$ and $\theta_{[T_{\Gamma}]}$ to be the restriction of $\theta$ to  $[V_{\Gamma}]= \{[v]: v \in V_\Gamma\}$ and $[T_{\Gamma}] = \{[t]: t\in T_\Gamma\}$, respectively. Together with four bijective  maps:
 \[\theta_{V_{\Gamma}}: V_{\Gamma} \rightarrow [V_{\Gamma}],\   v\mapsto [v]\ \mbox{for\ all}\ v \in V_{\Gamma}\,\]
 \[\theta_{V_{\Delta}}: [V_{\Delta}] \rightarrow V_{\Delta},\  [v]\mapsto v \ \mbox{for\ all}\ v \in V_{\Delta}\,\]
  \[\theta_{T_{\Gamma}}: T_{\Gamma} \rightarrow [T_{\Gamma}],\   t\mapsto [t]\ \mbox{for\ all}\ t \in T_{\Gamma}\ \]
  and
 \[\theta_{T_{\Delta}}: [T_{\Delta}] \rightarrow T_{\Delta},\   [t]\mapsto t\ \mbox{for\ all}\ t \in T_{\Delta},\ \]
 we get  that  $\overline{\theta} = (\theta_{V_{\Gamma}}\theta_{[V_{\Gamma}]}\theta_{V_{\Delta}}, \theta_{T_{\Gamma}}\theta_{[T_{\Gamma}]}\theta_{T_{\Delta}})$ is an  isomorphism from $\Gamma$ to $\Delta$.
\end{proof}

\section{An example}\label{sec:ex}
In this section we give an example of $Q_{\Gamma}$ contains properly $S_{\Gamma}$ and $R_{\Gamma}$.

Let $\Gamma = (V, T)$ be a   network  as shown in the following, where
$$V = \{v_1, v_2, v_3, v_4\} \ \mbox{and}\ T= \{t_1, t_2\},$$
 where $t_1=(\{v_1,v_2\},\{v_3\})$ and $t_2 = (\{v_3\}, \{v_4\})$.

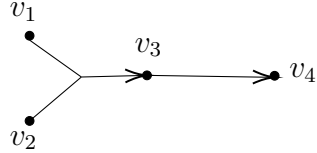
 \begin{center}
\begin{tikzpicture}[x=0.75pt,y=0.75pt,yscale=-1,xscale=1]

\draw  [fill={rgb, 255:red, 0; green, 0; blue, 0 }  ,fill opacity=1 ] (150,92.17) .. controls (150,90.95) and (150.99,89.96) .. (152.21,89.96) .. controls (153.43,89.96) and (154.42,90.95) .. (154.42,92.17) .. controls (154.42,93.39) and (153.43,94.37) .. (152.21,94.37) .. controls (150.99,94.37) and (150,93.39) .. (150,92.17) -- cycle ;
\draw  [fill={rgb, 255:red, 0; green, 0; blue, 0 }  ,fill opacity=1 ] (150,135.17) .. controls (150,133.95) and (150.99,132.96) .. (152.21,132.96) .. controls (153.43,132.96) and (154.42,133.95) .. (154.42,135.17) .. controls (154.42,136.39) and (153.43,137.37) .. (152.21,137.37) .. controls (150.99,137.37) and (150,136.39) .. (150,135.17) -- cycle ;
\draw    (178.33,113) -- (209.42,112.22) ;
\draw [shift={(211.42,112.17)}, rotate = 178.56] [color={rgb, 255:red, 0; green, 0; blue, 0 }  ][line width=0.75]    (10.93,-3.29) .. controls (6.95,-1.4) and (3.31,-0.3) .. (0,0) .. controls (3.31,0.3) and (6.95,1.4) .. (10.93,3.29)   ;
\draw  [fill={rgb, 255:red, 0; green, 0; blue, 0 }  ,fill opacity=1 ] (209.21,112.17) .. controls (209.21,110.95) and (210.2,109.96) .. (211.42,109.96) .. controls (212.64,109.96) and (213.63,110.95) .. (213.63,112.17) .. controls (213.63,113.39) and (212.64,114.37) .. (211.42,114.37) .. controls (210.2,114.37) and (209.21,113.39) .. (209.21,112.17) -- cycle ;
\draw  [fill={rgb, 255:red, 0; green, 0; blue, 0 }  ,fill opacity=1 ] (273.58,112.37) .. controls (273.58,111.16) and (274.57,110.17) .. (275.79,110.17) .. controls (277.01,110.17) and (278,111.16) .. (278,112.37) .. controls (278,113.59) and (277.01,114.58) .. (275.79,114.58) .. controls (274.57,114.58) and (273.58,113.59) .. (273.58,112.37) -- cycle ;
\draw    (211.42,112.17) -- (273.79,113.14) ;
\draw [shift={(275.79,113.17)}, rotate = 180.89] [color={rgb, 255:red, 0; green, 0; blue, 0 }  ][line width=0.75]    (10.93,-3.29) .. controls (6.95,-1.4) and (3.31,-0.3) .. (0,0) .. controls (3.31,0.3) and (6.95,1.4) .. (10.93,3.29)   ;
\draw    (152.21,94.37) -- (178.33,113) ;
\draw    (152.21,135.17) -- (178.33,113) ;

\draw (141,74) node [anchor=north west][inner sep=0.75pt]   [align=left] {$v_1$};
\draw (141,139) node [anchor=north west][inner sep=0.75pt]   [align=left] {$v_2$};
\draw (204,92) node [anchor=north west][inner sep=0.75pt]   [align=left] {$v_3$};
\draw (282,105) node [anchor=north west][inner sep=0.75pt]   [align=left] {$v_4$};
\end{tikzpicture}
 \end{center}
 According to the definition of $T^0$ in Section~\ref{subsec:HOLN}, we get
$$T^0=\{A\} \cup V,$$
where $A = \{v_1,v_2\}$.
Let $X=\{v_1, v_2, A\}$,
\[X_{v_1,v_2}= \{w \in X^+: v_1v_1,\ v_2v_2, AA, v_1v_2\ \mbox{and}\ v_2v_1 \ \mbox{are\ not \ subwords\ in}\ w\}\]
and
\[X_A= \{w \in X_{v_1, v_2}: w=\mu x, x\in X\setminus \{A\}\}.\]
 We  have
$$RP(\Gamma)=T \cup T^0 \cup X_{v_1,v_2} \cup  X_At_1 \cup X_At_1t_2 \cup \{t_1t_2\}$$
and
$$RLP(\Gamma)=T \cup T^0 \cup \{t_1t_2\}.$$

As every element of $Q_{\Gamma}$ has a uniquely right normal form representative, we list elements of $Q_{\Gamma}$ as follows:

\begin{align*}
	Q_\Gamma&=\{[\alpha \mathbf{r}(\alpha)]: \alpha \in RP(\Gamma)\}\cup \{[\mathbf{r}(\beta)\beta^{-1}]: \beta \in RLP(\Gamma)\}\\
	&\hspace{6mm}\cup \{[\alpha\beta^{-1}]: \alpha, \beta \in RLP(\Gamma)\ \mbox{and}\ \mathbf{r}(\alpha) = \mathbf{r}(\beta)\}\\
	&\hspace{6mm}\cup \{[\alpha t_1^{-1}]: \alpha \in X_AA \cup X_At_1\} \\
    &\hspace{6mm}\cup \{[\alpha t_2^{-1}]: \alpha \in X_At_1t_2\} \cup  \{[\alpha(t_1t_2)^{-1}]: \alpha \in X_At_1t_2\} \cup \{[0]\},
\end{align*}
\begin{align*}
	S_\Gamma&=\{[\mathbf{r}(\beta)\beta^{-1}]: \beta \in RLP(\Gamma)\setminus T^0\}\\
	&\hspace{6mm}\cup \{[\alpha\beta^{-1}]: \alpha \in RLP(\Gamma), \beta \in RLP(\Gamma)\setminus T^0\ \mbox{and}\ \mathbf{r}(\alpha) = \mathbf{r}(\beta)\}\\
	&\hspace{6mm}\cup \{[\alpha t_1^{-1}]: \alpha \in X_AA \cup X_At_1\} \\
    &\hspace{6mm}\cup \{[\alpha t_2^{-1}]: \alpha \in X_At_1t_2\} \cup  \{[\alpha(t_1t_2)^{-1}]: \alpha \in X_At_1t_2\} \cup \{[0]\},
\end{align*}

and
\begin{align*}
R_\Gamma&= \{[\alpha\beta^{-1}]: \alpha, \beta \in RLP(\Gamma)\setminus T^0\ \mbox{and}\ \mathbf{r}(\alpha) = \mathbf{r}(\beta)\} \cup \{[0]\}\\
&=\{[t_1t_1^{-1}],\ [t_2t_2^{-1}],\ [t_1t_2t_2^{-1}],\ [t_1t_2(t_1t_2)^{-1}],\ [0]\}.
\end{align*}

Clearly, $R_\Gamma\subset S_\Gamma \subset Q_\Gamma$.

Set
\begin{align*}
I_1 &= Q_{\Gamma}[t_2t_2^{-1}]Q_{\Gamma}=\{[\alpha\beta^{-1}]: \alpha \in RP(\Gamma), \beta \in RLP(\Gamma), \mathbf{r}(\alpha)=\mathbf{r}(\beta)=\mathbf{r}(t_2)\} \cup \{[0]\}\\
&= \{[\alpha\mathbf{r}(t_2)]: \alpha \in \{t_2, \mathbf{r}(t_2), t_1t_2\} \cup X_At_1t_2\} \cup \{[\mathbf{r}(t_2)\beta^{-1}]: \beta \in \{t_2, t_1t_2\}\}\\
&\hspace{6mm} \cup \{[\alpha\beta^{-1}]: \alpha, \beta \in \{t_2, t_1t_2\}\}\\
&\hspace{6mm} \cup \{[\alpha t_2^{-1}]: \alpha \in X_At_1t_2\} \cup  \{[\alpha(t_1t_2)^{-1}]: \alpha \in X_At_1t_2\} \cup \{[0]\},
\end{align*}

\begin{align*}
I_2 &= S^1_{\Gamma}[t_2t_2^{-1}]S^1_{\Gamma}=\{[\alpha\beta^{-1}]: \alpha \in RP(\Gamma), \beta \in RLP(\Gamma)\setminus T^0, \mathbf{r}(\alpha)=\mathbf{r}(\beta)=\mathbf{r}(t_2)\} \cup \{[0]\}\\
&= \{[\mathbf{r}(t_2)\beta^{-1}]: \beta \in \{t_2, t_1t_2\}\}\\
&\hspace{6mm} \cup \{[\alpha\beta^{-1}]: \alpha, \beta \in \{t_2, t_1t_2\}\}\\
&\hspace{6mm} \cup \{[\alpha t_2^{-1}]: \alpha \in X_At_1t_2\} \cup  \{[\alpha(t_1t_2)^{-1}]: \alpha \in X_At_1t_2\} \cup \{[0]\}
\end{align*}

 and
\begin{align*}
I_3&= R_{\Gamma}[t_2t_2^{-1}]R_{\Gamma}=\{[\alpha\beta^{-1}]: \alpha, \beta \in RLP(\Gamma)\setminus T^0, \mathbf{r}(\alpha)=\mathbf{r}(\beta)=\mathbf{r}(t_2)\} \cup \{[0]\}\\
&= \cup \{[\alpha\beta^{-1}]: \alpha, \beta \in \{t_2, t_1t_2\}\}\cup \{[0]\}.
\end{align*}
 Then by Lemma~\ref{*-ideal}, $I_1$ is a proper $*$-ideal of $Q_{\Gamma}$; by Lemma~\ref{*-ideal S}, $I_2$ is a proper $*$-ideal of $S_{\Gamma}$; by Lemma~\ref{*-ideal R}, $I_3$ is a proper ideal of $R_\Gamma$.

\section*{Declarations}
The authors are very grateful to Professor Victoria Gould  for her careful reading of the paper, and for her valuable suggestions, especially for rewriting systems. The research is supported by Natural Science Foundation of China (Grant No:11471255, 11501331). The first author was supported by the Cultivation Project of Young and Innovative Talents in Universities of Shandong Province.

\end{document}